\documentclass[english]{amsart}
\usepackage[utf8]{inputenc}
\usepackage{geometry}
\usepackage{babel}
\usepackage{mathtools}
\usepackage{multirow}
\usepackage{amstext}
\usepackage{amsthm}
\usepackage{amssymb}
\usepackage[unicode=true,pdfusetitle, 
bookmarks=true,bookmarksnumbered=false,bookmarksopen=false,
 breaklinks=false,pdfborder={0 0 1},backref=false,colorlinks=false]
 {hyperref}
 \usepackage{cleveref}
\makeatletter
\numberwithin{equation}{section}
\numberwithin{figure}{section}
\theoremstyle{plain}
\newtheorem{theorem}{\protect\theoremname}[section]
\theoremstyle{plain}
\newtheorem{lemma}[theorem]{\protect\lemmaname}
\theoremstyle{plain}
\newtheorem{cor}[theorem]{\protect\corollaryname}
\theoremstyle{plain}

\theoremstyle{plain}

\newtheorem{question}[theorem]{\protect\questionname}
\theoremstyle{plain}
\newtheorem{proposition}[theorem]{\protect\propositionname}
\theoremstyle{definition}
\newtheorem{definition}[theorem]{\protect\definitionname}
\theoremstyle{remark}

\theoremstyle{definition}
\newtheorem{example}[theorem]{\protect\examplename}
\newtheorem{thmx}{Theorem}

\usepackage{subcaption}
\usepackage{comment}
\usepackage{graphicx}
\usepackage{xcolor}
\usepackage{tikz}
\usepackage{pgfplots}
\usetikzlibrary{positioning}

\global\long\def\R{\mathbb{R}}%
\global\long\def\C{\mathbb{C}}%
\global\long\def\K{\mathbb{K}}%
\global\long\def\Z{\mathbb{Z}}%
\global\long\def\N{\mathbb{N}}%
\global\long\def\norm#1#2{\left\Vert #1\right\Vert _{#2}}%
\global\long\def\G{\mathbb{G}}%
\global\long\def\H{\mathbb{H}}%
\global\long\def\K{\mathbb{K}}%
\def \FF{\mathcal F}

\global\long\def\define#1{\textit{#1}}
\global\long\def\BS{\mathrm{BS}}
\makeatother

\global\long\def\Cay{\mathrm{Cay}}%
\global\long\def\Sch{\mathrm{Sch}}
\global\long\def\Stab{\mathrm{Stab}}

\providecommand{\corollaryname}{Corollary}
\providecommand{\definitionname}{Definition}
\providecommand{\examplename}{Example}
\providecommand{\lemmaname}{Lemma}
\providecommand{\propositionname}{Proposition}
\providecommand{\remarkname}{Remark}
\providecommand{\theoremname}{Theorem}
\providecommand{\questionname}{Question}

\usepackage{subcaption}
\usepackage{graphicx}
\usepackage{xcolor}
\usepackage{tikz}

\usetikzlibrary{shapes,arrows,automata,matrix,fit,patterns,positioning}
\usepackage{pgfplots}
\usepackage{stmaryrd}

\makeatother

\providecommand{\conjecturename}{Conjecture}
\providecommand{\corollaryname}{Corollary}
\providecommand{\definitionname}{Definition}
\providecommand{\examplename}{Example}
\providecommand{\lemmaname}{Lemma}
\providecommand{\propositionname}{Proposition}
\providecommand{\remarkname}{Remark}
\providecommand{\theoremname}{Theorem}

\usepackage{marvosym,fancybox}

\newcommand{\xConfig}[2]{%
	\begin{tikzpicture}[
		baseline=-\the\dimexpr\fontdimen22\textfont2\relax,ampersand replacement=\&]
		\matrix[
			matrix of math nodes,
			nodes={
				minimum size=1.2ex,text width=1.2ex,
				text height=1.2ex,inner sep=3pt,draw={gray!20},align=center,
				anchor=base
			}, row sep=1pt,column sep=1pt
		] (config) {#1};
		\node[draw,rectangle,dashed,help lines,fit=(config), inner sep=0.5ex] {};
		#2
	\end{tikzpicture}
}

\begin{document}

\title[Obstructions to self-simulability]{A geometric obstruction to self-simulation for groups}

\author{
	Sebasti\'an Barbieri, Kan\'eda Blot, Mathieu Sablik and Ville Salo
}

\newcommand{\Addresses}{{
		\bigskip

		\hskip-\parindent   S.~Barbieri, \textsc{Departamento de Matem\'aticas, Universidad de Chile.}\par\nopagebreak
		\textit{E-mail address}: \texttt{sbarbieri@uchile.cl}

        \medskip
		
		\hskip-\parindent   K.~Blot, \textsc{DER de math\'ematiques, ENS Paris-Saclay.}\par\nopagebreak
		\textit{E-mail address}: \texttt{caneda.blot@ens-paris-saclay.fr}
		
		\medskip
		
		\hskip-\parindent   M.~Sablik, \textsc{Universit\'e Paul Sabatier.}\par\nopagebreak
		\textit{E-mail address}: \texttt{mathieu.sablik@math.univ-toulouse.fr}
		
		\medskip
		
		\hskip-\parindent   V.~Salo, \textsc{University of Turku.}\par\nopagebreak
		\textit{E-mail address}: \texttt{vosalo@utu.fi}
}}

\begin{abstract}
We introduce a new quasi-isometry invariant for finitely generated groups and show that every group with this property admits a subshift which is effectively closed by patterns and that cannot be realized as the topological factor of any subshift of finite type. We provide several examples of groups with the property, such as amenable groups, multi-ended groups, generalized Baumslag-Solitar groups, fundamental groups of surfaces, and cocompact Fuchsian groups.
\medskip
		
		\noindent
		\emph{Keywords}: Symbolic dynamics, self-simulation, quasi-isometry, effective dynamics.
		
		\noindent
		\emph{MSC2020}: \textit{Primary:} 37B10. 
		\textit{Secondary:} 37B05,  
                20F65, 
                20F10. 

\end{abstract}

\maketitle
\tableofcontents
\section{Introduction}

Given a countable group $\G$ and a finite set $\Sigma$, a subshift is a closed subset of $\Sigma^\G$ under the prodiscrete topology that is invariant under the shift $\G$-action. Subshifts of finite type (SFTs) are an important class of subshifts which are defined through a finite number of local rules. Following the celebrated result of Hochman~\cite{Hochman-2009} that states that any computable $\Z^d$-action can be obtained as a factor of a subaction of a $\Z^{d+2}$-SFT, a number of results in the literature~\cite{Aubrun-Sablik-2010,Durand-Romashchenko-Shen-2010,Hochman-Meyerovitch-2010,Bar_2019_geometricsimulation,BarSabSal2025s-selfsimulablegroups} have explored how far are SFTs and their symbolic factors, called sofic subshifts, of encoding all computable dynamics of a given group.

In the setting of expansive actions, it is natural to ask which dynamics can be realized by sofic subshifts. One natural computability obstruction (for finitely generated and recursively presented groups) is that sofic subshifts are effective; that is, there exists a description of the subshift through a set of forbidden patterns that can be enumerated algorithmically. A series of studies have explored which dynamical properties of an effective $\Z^d$-subshift actually force it to become sofic, for instance, combinatorial constraints from $\Pi_1^0$-sets~\cite{Westrick-2017}, or very low pattern complexity~\cite{destombes2023algorithmiccomplexitysoficnessshifts}.
 
 Conversely, it is natural to explore which effective subshifts are not sofic. To the best of our knowledge, the most well-known example of an effective $\Z^2$-subshift that is not sofic is the mirror subshift described in~\cite{Hoch_chapter_CWSD}. This is the space of all $\Z^2$ configurations over $\{0,1,\star\}$ with the property that if one $\star$ occurs, it extends vertically and imposes a mirror symmetry on the remainder of the configuration, see~\Cref{fig:mirror}. The proof of non-soficity is based on a combinatorial argument that heavily relies on the amenability of the underlying group.

   \begin{figure}[ht!]
     \centering
     \begin{tikzpicture}[scale = 0.4]
         \foreach \i in {-1,...,11}{ 
            \foreach \j in {-7,...,7}{
                \draw[thick, fill=white] (\j,\i) rectangle (\j+1,\i+1);
                \node at (\j+0.5,\i+0.5) {$0$};
            }
         }
         
         \foreach \x/\y in{3/0,5/0,2/1,4/1,2/2,4/2,1/3,2/3,3/3,4/3,5/3,6/3,1/4,2/4,3/4,4/4,5/4,6/4,1/5,4/5,5/5,1/6,4/6,5/6,1/7,2/7,3/7,4/7,1/8,2/8,3/8,3/9,4/9,4/10}{
            \draw[thick, fill=black!30] (-\x,\y) rectangle (-\x+1,\y+1);
            \draw[thick, fill=black!30] (\x,\y) rectangle (\x+1,\y+1);
            \node at (\x+0.5,\y+0.5) {$1$};
            \node at (-\x+0.5,\y+0.5) {$1$};
        }
        \foreach \i in {-1,...,11}{ 
         \draw[thick, fill=black!10] (0,\i) rectangle (1,\i+1);
         \node at (0.5,\i+0.5) {$\star$};
         }
     \end{tikzpicture}
		 \caption{A portion of a configuration in the $\Z^2$-mirror shift.}
         \label{fig:mirror}
 \end{figure}
 
Variations of the argument used to prove non-soficity of the mirror shift have been used to demonstrate the non-soficity of a subshift corresponding to the superposition of Sturmian words of the same slope~\cite{Fernique-Sablik-2019} or minimal effective subshifts which have high entropy dimension~\cite{Gangloff-2022}. In the case of $\Z^d$, other sufficient conditions for non-soficity can be expressed in terms of extender sets~\cite{Kass-Maden-2013} and Kolmogorov complexity~\cite{Destombes-Romashchenko-2022}.
  
 In the context of finitely generated groups with decidable word problem, the mirror subshift argument has been used to construct effective subshifts that are not sofic when the group is amenable or multi-ended~\cite{AubBarSab_effectiveness_2017}. Hence a natural question is whether such a construction exists in all groups or if there are groups where all effective subshifts are sofic. In~\cite{BarSabSal2025s-selfsimulablegroups} the authors introduce the class of \define{self-simulable} groups, in which every effective subshift is sofic and show that it is non-trivial. We remark that while the definition of self-simulable is actually more general, see~\Cref{def:selfsimulable}, for finitely generated and recursively presented groups it is equivalent to the statement that all effective subshifts are sofic.

 In~\cite{BarSabSal2025s-selfsimulablegroups} it was shown that the direct product of any two finitely generated nonamenable groups is self-simulable. This, along with some stability properties of the class of self-simulable groups, was used to produce a list of interesting examples such as $\operatorname{SL}_n(\Z)$ for $n \geq 5$ and Thompson's $V$. It was also shown that Thompson's group $F$ is self-simulable if and only if it is nonamenable. 
 
 Two properties of a finitely generated group with decidable word problem that provide obstructions to being self-simulable are amenability and the property of having multiple ends. However, it was also shown in~\cite{BarSabSal2025s-selfsimulablegroups} that there exist groups which are nonamenable and $1$-ended that are not self-simulable, such as $F_2 \times \Z$. 
 
 The main purpose of this article is to present a new quasi-isometry invariant for finitely generated groups which provides an obstruction to self-simulability. Furthermore, all examples we are aware of that are not self-simulable satisfy this property. We call groups which satisfy this property \define{extraterrestrial} because their Cayley graphs admit structures which roughly resemble flying saucers.

Let us be more precise. Consider an infinite graph $G$ of bounded degree and non-negative integers $m,k,r$. An $(m,k,r)$-UFO is a triple of sets of vertices $(U,F,O)$ of $G$ which satisfies $|U| \geq m|F|$, that $U$ and $O$ can be matched by paths of length at most $k$, and that all paths from $U$ to $O$ which avoid vertices in $F$ have at least length $r$. A graph is called extraterrestrial if for all $m$ there is some $k$ such that for all $r$ then $G$ admits an $(m,k,r)$-UFO. In~\Cref{fig:ufo_in_the_grid} we represent a UFO in a Cayley graph of $\Z^2$.

\begin{figure}[ht!]
    \centering
    \begin{tikzpicture}[scale = 0.4]
        \draw[black!50] (0,0) grid (21,1);
        \draw[black!50] (7,1) grid (14,4);
        \draw[black!50] (7,-3) grid (14,0);
        \draw[black, very thick, fill = red!20, opacity = 0.3 ] (0,0) rectangle (21,1);
        \draw[black, very thick] (0,0) rectangle (21,1);
        \draw[black, very thick, fill = blue!20, opacity = 0.3] (7,1) rectangle (14,4);
         \draw[black, very thick] (7,1) rectangle (14,4);
        \draw[black, very thick, fill = green!20, opacity = 0.3] (7,-3) rectangle (14,0);
         \draw[black, very thick] (7,-3) rectangle (14,0);
        \node at (10.5,0.5) {$F$};
        \node at (10.5,2.5) {$O$};
        \node at (10.5,-1.5) {$U$};
    \end{tikzpicture}
    \caption{A $(1,4,18)$-UFO for the canonical Cayley graph of $\Z^2$.}
    \label{fig:ufo_in_the_grid}
\end{figure}

Intuitively, a graph $G$ is extraterrestrial if there are pairs of subsets of vertices $U,O$ of the same size, which are close (parametrized by $k$) in $G$, but that can be made to be arbitrarily far (parametrized by $r$) by removing a set of vertices $F$ which is small (parametrized by $m$) in comparison with $U$ and $O$.


Our first observation is that the property of being extraterrestrial is a geometric invariant:

\begin{thmx}\label{mainthm:property_ET_QI_invariant}
    Being extraterrestrial is a quasi-isometry invariant for graphs of bounded degree.
\end{thmx}

In particular, given a finitely generated group $\G$, it follows that if for some finite generating set $S$ of $\G$ the Cayley graph $\Cay(\G,S)$ is extraterrestrial, then the property is preserved by changing the set of generators. In particular, being extraterrestrial is a property of the group.

Our main result is that being extraterrestrial is an obstruction to self-simulability. 

\begin{thmx}\label{mainthm:property_ET_groups_not_SS}
    Let $\G$ be a finitely generated extraterrestrial group, there exists a $\G$-subshift which is effectively closed by patterns and which is not the topological factor of any $\G$-SFT.
\end{thmx}

Intuitively, a subshift is effectively closed by patterns if it can be described by a list of forbidden cylinders that are produced by a Turing machine. In the case where $\G$ is recursively presented, every subshift which is effectively closed by patterns is topologically conjugate to an expansive action by computable maps on a $\Pi_1^0$ subset of the Cantor space, thus in particular these groups cannot be self-simulable. In the case of groups which are non-recursively presented, this provides an obstruction to the natural relativization of self-simulability, which we call strong self-simulability, see~\Cref{subsec_computability_shift}.

The proof of~\Cref{mainthm:property_ET_groups_not_SS} relies on a construction that in a very abstract way generalizes the mirror shift. More precisely, we replace the algebraic global symmetry with respect to mirror symbols by global symmetry through matchings in the group induced by a lexicographic ordering. The property of being extraterrestrial provides the precise geometrical structure in the group that makes this construction work.

The last part of the article is dedicated to producing examples of extraterrestrial groups and tools to prove that a group is extraterrestrial. The following theorem summarizes the most relevant examples.

\begin{thmx}\label{mainthm:groups-with-ET}
    The following classes of finitely generated groups are extraterrestrial:
    \begin{itemize}
        \item Infinite amenable groups (\Cref{prop:AmenableET}).
        \item Multi-ended groups (\Cref{prop:MultiEndedET}).
        \item Hyperbolic groups with infinite outer automorphism group (\Cref{cor:hyp_outer_ext}).
        \item Amalgamated free products relative to an amenable subgroup with strict embeddings (\Cref{prop:amalgamated_free_ET}).
        \item HNN-extensions relative to an amenable subgroup (\Cref{prop:HNN_ET}).
        \item Generalized Baumslag-Solitar groups (\Cref{cor:BS_has_ET}).
        \item Cocompact Fuchsian groups (\Cref{cor:Fuchsian_ET}).
        \item Fundamental groups of surfaces (\Cref{cor:surfaces}).
    \end{itemize}
\end{thmx}

Furthermore, we show that Thompson's group $F$ is amenable if and only if it is extraterrestrial (\Cref{prop:thompsonF}).

In~\cite[Question 9.3]{BarSabSal2025s-selfsimulablegroups} we asked whether every $1$-ended hyperbolic group is self-simulable. The previous result answers this question negatively: the fundamental group of a closed orientable surface of genus $2$ is a $1$-ended hyperbolic group which is extraterrestrial and thus cannot be self-simulable.

The main tool used to provide the examples in~\Cref{mainthm:groups-with-ET} is that every group that admits an amenable subgroup whose associated Schreier graph is multi-ended is extraterrestrial (\Cref{thm:AmenableExtension}). We do not know any example of extraterrestrial group which does not fall into this case. We also do not know any example of finitely generated group which is not (strongly) self-simulable and fails to be extraterrestrial. 

\begin{question}
    Is there any group which is neither strongly self-simulable nor extraterrestrial?
\end{question}

\begin{question}
    Is there any group which is extraterrestrial but does not admit an amenable subgroup whose associated Schreier graph is multi-ended?
\end{question}



\textbf{Acknowledgments}:  S. Barbieri was supported by projects ANID FONDECYT regular 1240085, AMSUD240026 and ECOS230003.

\section{Preliminaries}

We shall use the notation $A\Subset B$ to denote that $A$ is a finite subset of $B$. We also use $A^*$ to denote the set of words on alphabet $A$.

Let $\G$ be a group with identity $1_{\G}$ and $S\Subset \G$ a symmetric set, that is, such that if $s\in S$ then $s^{-1}\in S$. The \define{word problem} of a group $\G$ with respect to $S \Subset \G$ is the language 
\[ \texttt{WP}_{S}(\G) = \{ w \in S^* :  \underline{w} = 1_{\G}   \} \]
where $\underline{w}$ denotes the element of $\G$ represented by the word $w$.

Now let $\G$ be a finitely generated group and $S\Subset G$ a finite generating symmetric set. We say that $\G$ is \define{recursively presented} (with respect to $S \Subset \G$) if $\texttt{WP}_{S}(\G)$ is recursively enumerable. A group is said to have \define{decidable word problem} (with respect to $S \Subset \G$) if $\texttt{WP}_{S}(\G)$ is decidable. These notions do not depend upon the choice of generating set $S$.

A \define{graph} is a pair $G = (V,E)$ where $V$ is a set of \define{vertices} and $E$ is a set of undirected \define{edges}, that is, of unordered pairs of elements of $V$. The degree $\mathrm{deg}(v)$ of a vertex $v\in V$ is the number of edges containing $v$. A graph has \define{bounded degree} if $\max_{v \in V}\mathrm{deg}(v)$ is finite.  A \define{path} of length $n \geq 0$ between $u$ and $v$ in $G$ is a finite sequence $w_0,\dots,w_n$ of distinct elements of $V$ such that $u=w_0$, $v=w_n$ and $\{w_i,w_{i+1}\}$ is an edge for all $i \in \{0,\dots,n-1\}$. We say $G$ is \define{connected} if there is path between every pair of $u,v\in V$. For a connected graph $G$ we denote by $d_G$ the metric on $V$ where $d_G(u,v)$ is the length of the shortest path from $u$ to $v$. If the graph is not connected and no such path exists between $u,v$ we set $d_G(u,v)=\infty$.

For a group $\G$ and a symmetric $S\Subset \G$, the \define{right Cayley graph} $\mathrm{Cay}(\G, S)$ is the graph where the vertices are $\G$ and there is an edge between $g$ and $g'$ if there exists $s\in S$ such that $g'=gs$. If $\G$ is generated by $S$ then $\mathrm{Cay}(\G, S)$ is a connected graph of bounded degree and we denote the metric $d_{\mathrm{Cay}(\G, S)}$ on $\G$ by $d_S$ instead. For an element $g \in G$ we denote $|g|_S = d_S(g,1_{\G})$ and for $k \in \N$ we write $B_k = \{g \in \G : |g|_S \leq k\}$ for the ball of radius $k$ on $\G$ under the metric $d_S$.

Similarly, for a subgroup $\H \leqslant \G$ and a symmetric $S\Subset \G$, the \define{right Schreier graph} $\Sch(\G,\H,S)$ is the graph where the vertices are $\{\H g, g \in \G\}$ and there is an edge between $\H g$ and $\H g'$ if there exists $s \in S$ such that $\H g' = \H g s$. It can equivalently be defined as the quotient graph of $\mathrm{Cay}(\G,S)$ by the left action of $\H$ by multiplication. 

\subsection{Shift spaces}

Let $\Sigma$ be a finite alphabet and $\G$ be a group. The full $\G$-shift is the set $\Sigma^\G = \{ x\colon \G \to \Sigma\}$ equipped with the left \define{shift} action $\G \curvearrowright \Sigma^{\G}$ by left multiplication given by 
\[ gx(h) = x(g^{-1}h) \qquad \mbox{  for every } g,h \in \G 
\mbox{ and } x \in \Sigma^\G.\]
The elements $a \in \Sigma$ and $x \in \Sigma^\G$ are called \define{symbols} and \define{configurations} respectively. We endow $\Sigma^\G$ with the prodiscrete topology generated by the clopen subbase given by the \define{cylinders} $[a]_g = \{x \in \Sigma^\G : x(g) = a\}$ where $g \in \G$. Given $F\Subset \G$, a \define{pattern} with support $F$ is an element $p \in \Sigma^F$. We denote the cylinder generated by $p$ by $[p] = \bigcap_{h \in F}[p(h)]_{h}$. 

In the case of a finitely generated group $\G$, it will be useful to encode patterns in a computable manner. Let $S$ be a symmetric set of generators and $\Sigma$ a finite alphabet. A \define{pattern coding} is a map $c\in \Sigma^W$ for some $W\Subset S^*$. The cylinder induced by $c$ is given by \[ [c] = \bigcap_{w \in W} [c(w)]_{\underline{w}}.  \]\
where $\underline{w}$ is the element of $\G$ represented by the word $w$. Notice that the cylinder associated to a pattern coding may be empty.

\begin{definition}
	A subset $X \subseteq \Sigma^\G$ is a \define{$\G$-subshift} if and only if it is $\G$-invariant and closed in the prodiscrete topology. 
\end{definition}

If the context is clear, we drop the $\G$ from the notation and speak plainly of a subshift. It follows directly from the definition that $X$ is a subshift if and only if there exists a set of patterns $\FF$ such that $X=X_{\FF}$, where \[X_\FF =  {\Sigma^\G \setminus \bigcup_{p \in \FF, g \in \G} g[p]}.\] 
If $X=X_{\FF}$, we say that $\FF$ is a set of \define{forbidden patterns} for $X$.

Let $X\subset (\Sigma_X)^{\G}$ and $Y\subset (\Sigma_Y)^{\G}$ be $\G$-subshifts. A \define{morphism} is a continuous function $\phi \colon X \to Y$ which is $\G$-equivariant, that is, $\phi(gx) = g \phi(x)$ for every $x \in X$ and $g \in \G$. A morphism is a \define{factor map} if the map is onto, and a \define{conjugacy} if it is bijective. Given a factor map $\phi\colon X \to Y$ we say that $Y$ is a factor of $X$, and that $X$ is an extension of $Y$.

Equivalently, by the Curtis-Lyndon-Hedlund theorem, see~\cite[Theorem 1.8.1]{ceccherini2010cellular}, a map $\phi\colon X \to Y$ is a morphism if and only if there exists $F\Subset \G$ and $\Phi \colon (\Sigma_X)^F \to \Sigma_Y$ such that $(\phi(x))(g) = \Phi(g^{-1}x|_F)$ for all $g \in \G$. We say that a morphism $\phi\colon X \to Y$ is a \define{$1$-block map} if there exists $\Phi\colon \Sigma_X \to \Sigma_Y$ such that $(\phi(x))(g) = \Phi(x(g))$ for every $g \in \G$ and $x\in X$. 

It well-known that a $\G$-action by homeomorphisms on a compact metrizable space $(X,d)$ of topological dimension zero is conjugate to some $\G$-subshift if and only if it is \define{expansive}, that is, there is a constant $C>0$ such that for every $x,y\in X$ then $x=y$ if and only if $d(gx,gy)<C$ for every $g \in \G$.

\begin{definition}
	Let $\G$ be a group and $S\Subset \G$ a symmetric set of generators of $\G$. We say that a subshift $X \subseteq \Sigma^\G$ is:
    \begin{enumerate}
    \item \define{$S$-nearest neighbor}, if $X = X_{\FF}$ for a finite set of forbidden patterns $\FF$. And the support of every pattern in $\FF$ is of the form $\{1_{\G},s\}$ for some $s \in S$. 
        \item a \define{subshift of finite type (SFT)}, if there exists a finite set of forbidden patterns $\FF$ such that $X = X_{\FF}$.
        \item a \define{sofic subshift}, if it is a factor of an SFT. 
    \end{enumerate}
    \end{definition}

We remark that all $S$-nearest neighbor subshifts are SFTs, and all SFTs are sofic subshifts. It is well known that every subshift of finite type is conjugate to an $S$-nearest neighbor subshift. Furthermore, if $X$ is a sofic subshift, it is always possible to construct an SFT $Y$ which is $S$-nearest neighbor and a $1$-block factor map $\phi$ such that $\phi(Y)= X$. For a proof of these facts, see~\cite{barbierilemp:tel-01563302}. 

\subsection{Computability of group actions}\label{subsec_computability_shift}

Next we will introduce notions of computability for group actions. We will assume the reader is acquainted with the basic concepts of computability, such as (oracle) Turing machines. An introduction can be found in~\cite{Cooper_2004_book}.

A subset $X\subset \{0,1\}^{\N}$ is called \define{effectively closed} or $\Pi_1^0$ if there exists a recursively enumerable language $L\subset \{0,1\}^*$ such that $X = \{0,1\}^{\N}\setminus \bigcup_{w \in L}[w]$, where $[w]$ denotes the cylinder set of all $x \in \{0,1\}^{\N}$ which begin by the word $w$. A map $f\colon X \to \{0,1\}^{\N}$ is  \define{computable} if there is an oracle Turing machine which on oracle $x \in X$ and input $n \in \N$ outputs $f(x)(n)$. An action of a finitely generated group $\G$ on a zero-dimensional topological space is called \define{effective} if it is topologically conjugate to a $\G$-action on an effectively closed set by computable maps.

\begin{definition}
    A $\G$-subshift is \define{effective} if it is conjugate to an (expansive) effective $\G$-action.
\end{definition}

It is often useful to have an explicit characterization of effective subshifts without passing through conjugacies. This can be done using pattern codings up to a small subtlety which we shall explain in a moment.

\begin{definition}
    A $\G$-subshift is {effectively closed by patterns}, if there exists a recursively enumerable set $\mathcal{C}$ of pattern codings such that $X = \Sigma^{\G}\setminus \bigcup_{c \in \mathcal{C}, g \in \G}g[c].$
\end{definition}

In the case where $\G$ is a recursively presented group, a $\G$-subshift is effectively closed by patterns if and only if it is an effective subshift~\cite[Corollary 7.7]{BarCarRoj2024} and in particular all sofic $\G$-subshifts are effective. Furthermore, if $\G$ has decidable word problem there is an enumeration of the group that makes the group operations computable and thus instead of considering pattern codings one can directly talk about recursively enumerable sets of forbidden patterns for the group.

\begin{definition}\label{def:selfsimulable}
    A finitely generated group $\G$ is \define{self-simulable} if every effective $\G$-action is the topological factor of a $\G$-SFT.
\end{definition}

It was shown in~\cite[Proposition 3.7]{BarSabSal2025s-selfsimulablegroups} that if $\G$ is a recursively presented self-simulable group, then $\G$ is nonamenable. Furthermore, if $\G$ is recursively presented and nonamenable, then $\G$ is self-simulable if and only if every effective $\G$-subshift is a factor of a $\G$-SFT~\cite[Theorem 4.6]{BarSabSal2025s-selfsimulablegroups}. Therefore for recursively presented groups being self-simulable coincides with the statement that the classes of sofic $\G$-subshifts and effective $\G$-subshifts coincide.

If $\G$ is not recursively presented, then for $|\Sigma|\geq 2$ even the full $\G$-shift $\Sigma^{\G}$ is non-effective, see~\cite[Proposition 7.9]{BarCarRoj2024}. It follows that the notion of self-simulation as stated above is not very meaningful beyond recursively presented groups. However, there is an easy way to ``fix'' the definition by relativizing the notion of an effective action.

Let $\G$ be a finitely generated group and $S$ a finite set of generators. Let $F(S)$ be the free group on basis $S$ and take $\psi\colon F(S)\to \G$ be the canonical epimorphism where for a reduced word $w \in F(S)$ we let $\psi(w)$ be the element obtained by multiplying the symbols of $w$ as elements in $\G$. We say that an action $\G\curvearrowright X$ is \define{relatively effective} if there exists an effective action $F(S)\curvearrowright Y$ such that if we let \[ X' = \{ y \in Y : gy=y \mbox{ for every } g \in \ker(\psi)\},   \]
then the actions $\G \curvearrowright X$ and $F(S)/\ker(\psi) \curvearrowright X'$ are topologically conjugate.

In the definition of relative effective action we have ``factored out'' the computational complexity of the word problem of $\G$. In fact, it is easy to see that for a recursively presented group $\G$ both notions coincide. It is natural then to consider self-simulation for non-recursively presented groups using this notion instead. This motivates the following definition.

\begin{definition}
    A finitely generated group $\G$ is \define{strongly self-simulable} if every relative effective $\G$-action is the factor of a $\G$-SFT.
\end{definition}

We remark that the main theorem of~\cite{BarSabSal2025s-selfsimulablegroups} which states that the direct product of any two finitely generated nonamenable groups is self-simulable, can be directly extended to show that in fact it is strongly self-simulable. In retrospective, we believe that strong self-simulability is the right notion for non-recursively presented groups, and thus in fact the right notion for all finitely generated groups.

The next proposition shows that $\G$-subshifts which are effectively closed by patterns are precisely the expansive relative effective $\G$-actions on spaces with zero topological dimension. Let $\G$ be finitely generated by $S$ and take $\psi\colon F(S) \to G$ the canonical epimorphism. Note that every $y \in \Sigma^{F(S)}$ such that $ty = y$ for every $t \in \operatorname{ker}(\psi)$ induces a configuration $x=\psi^*(y)\in \Sigma^{\G}$ given by \[ \psi^*(y)(g) = y(t) \mbox{ for every } g \in \G, \mbox{ where } t \in \psi^{-1}(g). \]
Conversely, every $x \in \Sigma^{\G}$ induces a map $y \in \Sigma^{F(S)}$ given by $y(g)=x(\psi(g))$ which is stabilized by $\operatorname{ker}(\psi)$. Hence there is a natural correspondence between configuration in $\Sigma^{\G}$ and configurations in $\Sigma^{F(S)}$ which are stabilized by $\operatorname{ker}(\psi)$.

\begin{proposition}\label{prop_caractES}
 A $\G$-subshift $X$ is effectively closed by patterns if and only if there exists an effective $F(S)$-subshift $Y$ such that \[ X = \{ \psi^*(y) : y \in Y \mbox{ and } gy=y \mbox{ for every } g \in \operatorname{ker}(\psi)\}. \]
\end{proposition}

\begin{proof}
    Let $\mathcal{C}$ be a set of pattern codings described over a finite generating set $S$ of $\G$. First notice that it can be interpreted as a set of pattern codings both on $\G$ and $F(S)$. Furthermore, note that if $x\in \Sigma^{\G},y\in \Sigma^{F(S)}$ are such that $x = \psi^*(y)$, then \[x \in \Sigma^{\G}\setminus \bigcup_{c\in \mathcal{C},g \in \G}g[c] \mbox{ if and only if }y \in \Sigma^{F(S)}\setminus \bigcup_{c\in \mathcal{C},t \in F(S)}t[c].\]

    As $F(S)$ is recursively presented, then effective $F(S)$-subshifts coincide with effectively closed by patterns $F(S)$-subshifts. It follows that if we take an effectively closed by patterns $X\subset \Sigma^{\G}$ and let $\mathcal{C}$ be a recursively enumerable set of pattern codings that defines it, then taking $Y\subset \Sigma^{F(S)}$ as the subshift which is defined by the same set of pattern codings in $F(S)$ gives the desired equality. Conversely, if one starts with an effective subshift $Y\subset \Sigma^{F(S)}$ then the subshift $X$ defined by the equality above is effectively closed by patterns.
\end{proof}

We note that if $\G$ is a finitely generated non-amenable group, then every relatively effective $\G$-action is the topological factor of an effectively closed by patterns $\G$-subshift (the proof is exactly the same as in~\cite[Theorem 4.6]{BarSabSal2025s-selfsimulablegroups}). Therefore for a non-amenable group $\G$, it follows that $\G$ is strongly self-simulable if and only if the classes of sofic $\G$-subshifts and effectively closed by patterns $\G$-subshifts coincide. Later we will show that no amenable group is strongly self-simulable (\Cref{prop:AmenableET} and~\Cref{mainthm:property_ET_groups_not_SS}) and thus the previous characterization stands for all finitely generated groups.

\section{UFOs and extraterrestrial graphs}

Let $G=(V,E)$ be a graph. Given two finite subsets $A,B$ of $V$, a \define{matching} is a relation $M\subset A \times B$ such that the projection maps to each component are injective. A matching $M$ is called \define{complete} if it induces a bijection (i.e, every $a\in A$ and $b \in B$ occurs in a unique element of $M$). We say that a matching $M$ is by \define{paths of length at most $k$} if for every $(u,v) \in M$ we have $d_G(u,v)\leq k$.

\begin{definition}
 Let $G=(V,E)$ be a graph. An $(m, k, r)$-UFO is a triple $(U,F,O)$ of disjoint finite subsets of $V$ such that:
\begin{enumerate}
\item \label{it:large} $|U|\geq m|F|$;
\item \label{it:matching} there is a complete matching between $U$ and $O$ by paths of length at most $k$;
\item \label{it:disconnect} any path from $U$ to $O$ in $G$ goes through a vertex in $F$ or is of length at least $r$.
\end{enumerate}
\end{definition}

\begin{definition}
 We say that a graph $G$ is \emph{extraterrestrial}, if for all $m \in \mathbb{N}$, there exists $k \in \mathbb{N}$, such that for all $r \in \mathbb{N}$, there exists an $(m, k, r)$-UFO.
\end{definition}

The acronym ``UFO'' means Under, Front and Over and is inspired by~\Cref{fig:ufo_in_the_grid} where a UFO is drawn on the canonical Cayley graph of $\Z^2$. We remark that this is not equivalent to the notion of UFO used in~\cite{FriSewZuc2024_minimalcharacteristic}. 
 
\begin{example}
Consider the Cayley graph of $\Z^d$ generated by the canonical basis $S= \{v \in \Z^d : \norm{v}{1}  =1\}$. For positive integers $m,r$ put
\begin{itemize}
\item $U=[0,r-1]^{d-1}\times[-3^{d-1} m,-1]$
\item $F=[-r,2r-1]^{d-1}\times\{0\}$
\item $O=[0,r-1]^{d-1}\times[1,3^{d-1} m]$
\end{itemize}
The triple (U,F,O) forms a $(m, 3^{d-1} m + 1, 2r+4)$-UFO, see~\Cref{fig:ufo_in_the_grid} for the case $m = 1, d = 2, r = 7$. We deduce that $\mathrm{Cay}(\Z^d,S)$ is extraterrestrial.
\end{example}

Next we will show that being extraterrestrial is an invariant of quasi-isometry for graphs of bounded degree.

\begin{definition}
 Two graphs $G=(V,E)$ and $G'=(V',E')$ are \define{quasi-isometric} if there exists a function $f\colon V\to V'$ and positive real constants $A$, $B$ and $C$ such that
 \begin{itemize}
\item $f$ is a \define{quasi-isometric embedding}: for every $v_1,v_2\in V$ one has
\[\frac{1}{A}d_G(v_1,v_2)-B\leq d_{G'}(f(v_1),f(v_2))\leq Ad_G(v_1,v_2)+B.\]
\item $f(V)$ is \define{relatively dense}: for every $v'\in V'$ there exists $v\in V$ such that 
\[d_{G'}(v',f(v))\leq C.\]
\end{itemize}
\end{definition}

Next we give the proof of~\Cref{mainthm:property_ET_QI_invariant}, that is, that the property of being extraterrestrial is an invariant of quasi-isometry for graphs of bounded degree.

\begin{proof}[Proof of~\Cref{mainthm:property_ET_QI_invariant}]
 Let $G=(V,E)$ and $G'=(V',E')$ be quasi-isometric graphs of bounded degree. Suppose that $G$ is extraterrestrial and let $f\colon V\to V'$ be a map which realizes the quasi-isometry between $G$ and $G'$ with associated constants $A,B$ and $C$ as in the definition. 

 Let $D$ be the largest of the degrees of $G$ and $G'$. It follows that the number of elements in a ball of radius $n$ centered on any vertex on any of both graphs is at most $D^{n}$. In particular, $D^{AB}$ is an upper bound of the number of preimages of a vertex by $f$, as two vertices $u,v \in V$ with the same image must satisfy $d_G(u,v) \leq AB$.
 
 Let $(U,F,O)$ be an $(m,k,r)$-UFO of $G$. Define 
 \[F'=\{v'\in V' : d_G(v' ,f(F))\leq \alpha\}\textrm{ with } \alpha= A^2(1+B+2C)+B+C.\] Define also $U''=f(U)\setminus F'$ and $O''=f(O)\setminus F'$. Consider a complete matching $M\subset U\times O$ between $U$ and $O$ by paths of length at most $k$. Take a maximal subset $M'' \subset M$ with the property that for any pair $(u_1,o_1), (u_2,o_2) \in M''$, then $\min(d_G(u_1,u_2),d_G(o_1,o_2)) > AB$. Define \[ M' = \{ (f(u),f(o)) : (u,o) \in M'' \mbox{ and }  (f(u),f(o))\in U''\times O''\}.   \]

 Notice that by the property we required on $M''$ the map $f$ is injective in both of its projections and thus it follows that $M'$ is indeed a matching which satisfies $|M'| \geq |M''|-2|F'|$.

 Finally, define $U'$ and $O'$ as the projections of $M'$ onto the first and second coordinates, respectively. With this definition, $M'$ is a complete matching between $U'$ and $O'$. Our aim is to prove that for $m$ large enough, the sets $(U', F', O')$ form an  $(m',k',r')$-UFO for some $(m',k',r')\in\N^3$.

 \begin{enumerate}
     \item An element of $V'$ has at most $D^\alpha$ elements at distance at most $\alpha$ and so $|F'|\leq|F|D^{\alpha}$. Moreover, as each vertex in $V'$ can have at most $D^{AB}$ preimages under $f$, it follows that $|M''| \geq \frac{|M|}{D^{2AB}} \geq \frac{m|F|}{D^{2AB}}$.

     By definition of $U'$ and the above considerations, we have \[|U'|=|M'|\geq |M''|-2|F'| \geq \frac{m|F|}{D^{2AB}}-2|F'| = \left( \frac{m}{D^{2AB+\alpha}}-2 \right)|F'| \geq m'|F'|.\] 

Where $m'=\left\lfloor\frac{m}{D^{2AB+\alpha}}-2 \right\rfloor$
\item Given two adjacent vertices $v_1,v_2\in V$, one has 
\[d_{G'}(f(v_1),f(v_2))\leq Ad_G(v_1,v_2)+B=A+B.\] It follows that if there is a path of length less than $k$ between $v_1$ and $v_2$, then there is path of length less than $k(A+B)$ between $f(v_1)$ and $f(v_2)$. Put $k'=\lceil k(A+B)\rceil$, we deduce that the set $M'$ is a complete matching between $U'$ and $O'$ by paths of length at most $k'$. 

\item Assume that there is a path $v'_0,v'_1,\dots,v'_{\ell}$ that avoids $F'$ of length ${\ell}$ in $G'$ between $f(v_0)=v'_0\in U'$ and $f(v_{\ell})=v'_{{\ell}}\in O'$. Given two consecutive vertices of this path, there exists $v_i$ and $v_{i+1}$ such that $f(v_i)$ and $f(v_{i+1})$ are at distance at most $C$ of the two consecutive vertices $v'_i$ and $v'_{i+1}$, so $d_{G'}(f(v_i),f(v_{i+1}))\leq 2C+1$. We deduce that 
\[d_G(v_i,v_{i+1})\leq A(d_{G'}(f(v_i),f(v_{i+1}))+B)\leq A((2C+1)+B).\]

Thus there is a path in $G$ between $v_i$ and $v_{i+1}$ of size at most $A(1+B+2C)$. Consider a vertex $v$ of this path, one has

\begin{align*}
    d_{G'}(v'_i,f(v))& \leq C+d_{G'}(f(v_i),f(v))\\  &\leq C+ Ad_G(v_i,v)+B\\ &\leq C+A(A(1+B+2C))+B\\ & =\alpha.
\end{align*}

Since $v'_i\notin F'$, one has $d_{G'}(v'_i,f(F))> \alpha$ so $v$ is not in $F$. We deduce that there exists a path from $v_0$ to $v_{\ell}$ which avoids $F$ and whose length is bounded by ${\ell}A(1+B+2C)$. From here we obtain that ${\ell}A(1+B+2C)\geq r$, thus if we let $r'=\left\lceil\frac{r}{A(1+B+2C)}\right\rceil$ it follows that every path in $G'$ between elements of $O'$ and $U'$ which avoids $F'$ has length at least $r'$.
 \end{enumerate}

Now we can prove the $G'$ is extraterrestrial. For $m'\in\N$, we choose $m\in \N$ such that $m'\leq \left\lfloor\frac{m}{D^{2AB+\alpha}}-2 \right\rfloor$. Consider $k\in \N$ such that $G$ admits an $(m,k,r)$-UFO for any $r\in\N$ and put $k'=\lceil k(A+B)\rceil$. Let $r'\in\N$, consider $r\in \N$ such that $r'\leq \left\lceil\frac{r}{A(1+B+2C)}\right\rceil$. Since $G$ admits a $(m,k,r)$-UFO, we deduce that $G'$ admits a $(m',k',r')$-UFO. So $G'$ is extraterrestrial
\end{proof}

If $\G$ is a finitely generated group and $S,S'$ are two finite symmetric sets of generators, then $\Cay(\G,S)$ and $\Cay(\G,S')$ are quasi-isometric. It follows that $\Cay(\G,S)$ is extraterrestrial if and only if $\Cay(\G,S')$ is extraterrestrial. This motivates the following definition.

\begin{definition}
    A group $\G$ is extraterrestrial if $\Cay(\G,S)$ is extraterrestrial for some (equivalently, any) symmetric finite set of generators $S$.
\end{definition}

Similarly, we remark that if $\H\leqslant \G$ is a subgroup and $S,S'$ are two symmetric finite set of generators of $\G$, then the Schreier graphs $\Cay(\G,\H,S)$ and $\Cay(\G,\H,S')$ are quasi-isometric, and thus we can speak about the pair $(\H,\G)$ being extraterrestrial.

\section{Extraterrestrial groups are not self-simulable}

In this section, we will show~\Cref{mainthm:property_ET_groups_not_SS}. We will first begin with a high-level sketch of the proof to illustrate the main ideas. Fix a finite symmetric set $S$ of generators for $\G$. Also, fix a well-order $1_{\G}<g_1<g_2<\dots$ on the elements of $\G$.

Consider first the alphabet $A= \{\star, \mathfrak{u},\mathfrak{o}\}$. For every $x\in A^{\G}$, we let $\mathcal{U}(x) = \{g \in \G : x(g) =\mathfrak{u}\}$ and $\mathcal{O}(x) = \{g \in \G : x(g) =\mathfrak{o}\}$ and define a unique greedy matching $M(x)\subset \mathcal{U}(x)\times \mathcal{O}(x)$ in the following manner: we first match all pairs of the form $(h,hg_1)$, then all pairs of the form $(h,hg_2)$ which were not previously matched and so forth. An example for $\G=\Z^2$ is shown in~\Cref{fig:highlevel}

\begin{figure}[ht!]
    \centering
    \begin{tikzpicture}[scale = 0.4]
    \begin{scope}
        \foreach \i in {1,...,8}{ 
            \foreach \j in {-4,...,4}{
                \draw[thick, fill=red!10] (\j,\i) rectangle (\j+1,\i+1);
                \node at (\j+0.5,\i+0.5) {$\star$};
            }
         }
         \foreach \i/\j in {2/3,4/4,2/-4, 6/-3, 7/-1, 6/3, 4/2, 4/-1}{
             \draw[thick, fill=green!10] (\j,\i) rectangle (\j+1,\i+1);
                \node at (\j+0.5,\i+0.5) {$\mathfrak{u}$};
         }
          \foreach \i/\j in {3/3,2/2,3/-3, 4/-2, 5/-1, 6/4, 4/3, 6/1}{
             \draw[thick, fill=blue!10] (\j,\i) rectangle (\j+1,\i+1);
                \node at (\j+0.5,\i+0.5) {$\mathfrak{o}$};
         }
    \end{scope}
     \begin{scope}[shift = {(11,0)}]
    \foreach \i in {1,...,8}{ 
        \foreach \j in {-4,...,4}{
            \draw[thick, fill=white] (\j,\i) rectangle (\j+1,\i+1);
        }
     }

     \foreach \i/\j in {2/3,4/4,2/-4, 6/-3, 7/-1, 6/3, 4/2, 4/-1}{
         \draw[thick, fill=black!20] (\j,\i) rectangle (\j+1,\i+1);
     }
      \foreach \i/\j in {3/3,2/2,3/-3, 4/-2, 5/-1, 6/4, 4/3, 6/1}{
         \draw[thick, fill=black!20] (\j,\i) rectangle (\j+1,\i+1);
     }
     \draw[blue, very thick, ->] (3.5,2.5) to (2.5,2.5);
     \draw[blue, very thick, ->] (4.5,4.5) to (3.5,4.5);
     \draw[blue, very thick, ->] (2.5,4.5) to (3.5,3.5);
     \draw[blue, very thick, ->] (3.5,6.5) to (4.5,6.5);
     \draw[blue, very thick, ->] (-3.5,2.5) to (-2.5,3.5);
     \draw[blue, very thick, ->] (-2.5,6.5) to (1.5,6.5);
     \draw[blue, very thick, ->] (-0.5,4.5) to (-1.5,4.5);
     \draw[blue, very thick, ->] (-0.5,7.5) to (-0.5,5.5);
    \end{scope}
     \end{tikzpicture}
    \caption{On the left we show $x \in A^{\Z^2}$ and the matching $M(x)$ is shown on the right. Here we have taken the order which first minimizes the $\ell^1$ norm and solves ties by minimizing the first coordinate and then the second.}
    \label{fig:highlevel}
\end{figure}

Now consider the alphabet $\Sigma = \{\star, \mathfrak{u},\mathfrak{o}\}\times \{-,+\}$. We identify elements of $\Sigma^{\G}$ as pairs of configurations $(x,y)\in \{\star, \mathfrak{u},\mathfrak{o}\}^{\G}\times \{-,+\}^{\G}$. We define the subshift $X\subset \Sigma^{\G}$ which consists of all pairs $(x,y)$ such that for every $(g,g')\in M(x)$ we have $y(g)=y(g')$, that is, matched elements must coincide in the second coordinate. It can be shown that if the well-order is computable (which can be done in a group with decidable word problem), then $X$ is effective.

We will show that if $\G$ is extraterrestrial, then $X$ cannot be a sofic subshift. Indeed, if $X$ were sofic, it would admit an SFT extension $Z$ which can be supposed to be nearest neighbor. By taking $m$ large enough and a suitable sequence of $(m,k,r)$-UFOs $(U_r,F_r,O_r)$, we can construct a sequence of configurations $(z_r,z'_r)\in Z^2$ indexed by $r$, which coincide in the $F_r$ portion of the UFO and such that their projections onto $X$ satisfy that the $U_r$ and $O_r$ portions are precisely those marked by $\mathfrak{u}$ and $\mathfrak{o}$ respectively. Furthermore, by an elementary counting argument, we can take them in such a way that $z_r$ and $z'_r$ differ in some matched pair. The contradiction is then obtained by appropriately taking a limit point $(z,z')\in Z^2$ of shifts of this sequence in such a way that after projecting to $X$, the projections of $z$ and $z'$ differ in a matched symbol and the components marked by $\mathfrak{u}$ and $\mathfrak{o}$ coincide. Furthermore, we ask that these two components are separated by an infinite set where $z$ and $z'$ coincide. Using the fact that $Z$ is nearest neighbor one can then construct a new configuration on $Z$ which projects to a configuration that violates the conditions in $X$, raising a contradiction.

This construction works in the case of a group $\G$ with decidable word problem. However, in the general case we need to introduce an extra layer that ``guesses the word problem of $\G$'' in order to ensure that our subshift is effectively closed by patterns. 

\subsection{A technical lemma}

In order to make our construction work for an arbitrary finitely generated group, we will require a technical lemma that shows that for every set of generators $S$ of a group $\G$ there is a configuration $x \in \{0,1\}^{\G}$ that distinguishes locally all elements of $\G$ by their neighborhoods, furthermore, the size of said neighborhood is linear with respect to the distance between the elements induced by $S$.

In the next proof we will use the fact that in any infinite and finitely generated group $\G$, if $S$ is a finite generating set and $B_{k}$ denotes the ball of radius $k$ in $\Cay(\G,S)$, then for any integer $m \geq 1$ we have $|B_{mk}|\geq \frac{m}{3}|B_k|$. This follows from considering a two-sided geodesic in $\Cay(\G,S)$ which passes through the origin and noting that one can fit $\lfloor \frac{2mk+1}{2k+1}\rfloor\geq \frac{m}{3}$ disjoint copies of $B_k$ within $B_{mk}$ whose centers lie on this geodesic.

\begin{lemma}\label{lemma:distinct_nbhds}
 There exists a universal constant $\Upsilon \geq 1$ such that for every infinite and finitely generated group $\G$, and every finite generating set $S$, there exists a configuration $\xi \in \{0,1\}^{\G}$ such that for every integer $k \geq 1$ and $g,h\in \G$ with $d_S(g, h) = k$, there is $t\in \G$ with $|t|_S \leq k\Upsilon$ such that $\xi(gt)\neq \xi(ht)$.
\end{lemma}

\begin{proof}
It is shown within the proof of~\cite[Theorem 2.4]{AubBarTho_aperiodic_2019} that if $(s_i)_{i \geq 1}$ is an enumeration of $\G \setminus \{1_G\}$, $C> 0$ is sufficiently large and $(T_i)_{i \geq 1}$ is a sequence of finite subsets of $\G$ such that $|T_i| \geq C i$ and $T_i \cap s_i T_i = \varnothing$, then there is a configuration $\xi$ where for each $g\in \G$ and each $i \geq 1$ there is $t=t(g,i)\in T_i$ such that $\xi(gt)\neq \xi(gs_it)$. Take $\Upsilon = 9\lceil C \rceil$ and fix some set of generators $S$. It suffices to show that we can choose  an enumeration $(s_i)_{i \geq 1}$ of $\G \setminus \{1_G\}$ and a sequence of finite subsets $T_i$ with the properties above and such that $T_i \subset B_{\Upsilon|s_i|_S}$ for all $i \geq 1$. 

Fix a sufficiently large constant $C$ and an enumeration $(s_i)_{i \geq 1}$ of $\G \setminus \{1_G\}$ in such a way that if $i\leq j$ then $|s_i|_S \leq |s_j|_S$. In this way we are sure that whenever $|s_i|_S=k$, then $i \leq |B_k|$.

For $i \geq 1$ denote $k = |s_i|_S$. We take $T_i$ as any maximal (for inclusion) subset of $B_{k\Upsilon}$ with the property that $T_i \cap s_i T_i = \varnothing$. We claim that $|T_i| \geq \frac{\Upsilon}{9}|B_k|$ and thus that $|T_i|\geq Ci$.

Indeed, suppose that $9|T_i|< \Upsilon|B_k|$. As $\G$ is infinite, we have $|B_{k\Upsilon}|\geq \frac{\Upsilon}{3}|B_k|$ from which it follows that $3|T_i| < |B_{k\Upsilon}|$ and thus that there exists $g\in B_{k\Upsilon}\setminus (s_i^{-1}T_i \cup T_i \cup s_iT_i)$. Take $T_i' = T_i \cup \{g\}$. By maximality of $T_i$ we have that $T'_i \cap s_iT' _i \neq \varnothing$, hence there are $h,h'  \in T'_i$ such that $h = s_ih'$. Clearly $h \neq h'$ as $s_i \neq 1_{\G}$. Furthermore, as $T_i \cap s_i T_i = \varnothing$, it follows that either $h = g$ or $h' = g$, but neither of these cases can hold because of our choice that $g\notin (s_i^{-1}T_i \cup T_i \cup s_iT_i)$.
\end{proof}

We note that in~\cite{AubBarTho_aperiodic_2019} the authors show that any value $C\geq 17$ works for the argument above, and in~\cite{MatthieuPipi_LLL} this bound is improved to $C\geq 11$.  In particular, the value $\Upsilon = 99$ works as the universal constant from~\Cref{lemma:distinct_nbhds}. We also remark that here we make no effort to optimize this constant and that with more careful handling of the inequalities it can certainly be improved.

\subsection{The generalized mirror shift}

We will construct, for every finitely generated group $\G$, a $\G$-subshift $X_{\texttt{GM}}$ which is effectively closed by patterns but not sofic whenever $\G$ is extraterrestrial. We call $X_{\texttt{GM}}$ the \define{generalized mirror subshift} because it is a geometric generalization of the classical mirror shift construction from \cite{Hoch_chapter_CWSD} discussed in the introduction. Variants of the mirror shift on nonabelian groups appear also in \cite{AubBarSab_effectiveness_2017,BarSabSal2025s-selfsimulablegroups}. 

For the remainder of this section, fix a finitely generated group $\G$ and a finite symmetric set of generators $S$ such that $1_{\G}\notin S$. Take $F(S)$ the free group generated by $S$ and let $\psi\colon F(S)\to \G$ be the canonical epimorphism. We also fix an arbitrary order $\leq$ on $S$ and consider the total order $\preceq$ on $S^*$ where words are first ordered by length and then lexicographically, that is: $s_1s_2\dots s_n\prec s'_1s'_2\dots s'_m$ if and only if either $n<m$ or $n=m$ and for some $k<n$, $s_1s_2\dots s_k = s'_1s'_2\dots s'_k$ and $s_{k+1}< s'_{k+1}$.

Fix an integer $\Upsilon$ as in~\Cref{lemma:distinct_nbhds}, consider an arbitrary alphabet $\Sigma$ and take $\Lambda = \Sigma \times \{0,1\}$. We will first define two types of patterns on alphabet $\Lambda$ in $F(S)$ which will be used in the definition of $X_{\texttt{GM}}$. For a reduced word $g \in F(S)$ denote by $|g|$ its word length. Furthermore, for $T\subset F(S)$ and $p \in \Lambda^T$, denote by $p_{\Sigma} \in \Sigma^T$ and $p_{\{0,1\}} \in \{0,1\}^T$ its projections, that is, the maps such that for every $t \in T$ we have $p(t)= (p_{\Sigma}(t),p_{\{0,1\}}(t))$.

For a positive integer $k$, let $B_k = \{g \in F(S) : |g|\leq k\}$. For a pattern $p \in \Lambda^{B_{(\Upsilon+1)k}}$ and $g,h \in B_k$ we say that $g,h$ are $p$-equivalent and write $g \simeq_p h$ if for every $t \in B_{k\Upsilon}$, $p_{\{0,1\}}(gt) = p_{\{0,1\}}(ht)$.

\begin{definition}
    A pattern $p \in \Lambda^{B_{(\Upsilon+1)k}}$ is \define{coherent} if for every pair of $p$-equivalent $g,h \in B_k$ we have $ p_{\Sigma}(g) = p_{\Sigma}(h)$.
\end{definition}

Intuitively, we use the $\{0,1\}^{F(S)}$ component as an oracle for an equivalence relation on $B_k$ where two elements $g,h \in B_k$ are identified if the symbols in $gB_{k\Upsilon}$ and $hB_{k\Upsilon}$ coincide. The fundamental observation is that~\Cref{lemma:distinct_nbhds} ensures that some configuration must encode the equivalence relation on $F(S)$ given by the word problem of $\G$.

Now we choose $\Sigma$ explicitly, namely, we take
\[\Sigma = \{\star, \mathfrak{u},\mathfrak{o}\}\times \{-,+\}.\]

For $T \subset F(S)$ and $p\in \Lambda^T$, we denote by $p_{\Sigma,A}$ and $p_{\Sigma,\pm}$ the projections of $p_{\Sigma}$ to each component of $\Sigma$ respectively. Furthermore, write $\mathcal{U}(p)$ (respectively $\mathcal{O}(p)$) for the set of positions $g\in T$ where $p_{\Sigma,A}(g)=\mathfrak{u}$ (respectively $\mathfrak{o}$).

For a coherent pattern $p \in \Lambda^{B_{(\Upsilon+1)k}}$, we define a matching $M(p)\subset F(S) \times F(S)$ with paths of length at most $k$ inductively. Let $w_1,\dots,w_m$ be the sequence of elements of $B_k\setminus \{1_{F(S)}\}$ of length at most $k$ ordered according to $\preceq$. Take $M_0(p)= \varnothing$ as the initial matching and $A_0(p) = (\mathcal{U}(p) \cup\mathcal{O}(p)) \cap B_k$ as the initial set of available vertices. 

Suppose that $M_{i-1}(p)$ and $A_{i-1}(p)$ have been constructed for some $1 \leq i\leq m$. We set $M_i(p)=M_{i-1}(p)$, $A_i(p)=A_{i-1}(p)$ and modify them as follows:

\begin{itemize}
    \item First, we remove from $A_i(p)$ all $g \in B_k$ such that either $g\in\mathcal{U}(p)$ but  $gw_{i}\notin B_k$ or $g\in\mathcal{O}(p)$ but $gw_i^{-1}\notin B_k$.
    \item In lexicographic order: for every $g,h \in A_i(p)$ such that $h = gw_i$ with $g \in \mathcal{U}(p)$ and $h\in\mathcal{O}(p)$, we add $(g,h)$ to $M_i(p)$ and remove from $A_i(p)$ all $t$ such that either $t\simeq_p g$ or $t\simeq_p h$.
\end{itemize}

Finally, we set $M(p)=M_m(p)$. The intuition behind this procedure is that elements from the quotient graph $B_k/\simeq_p$ are matched from $\mathcal{U}(p)$ to $\mathcal{O}(p)$ prioritizing according to the order on $S^*$. If at some point some element could have potentially been matched with an element outside of $B_k$, we remove it from the set of available vertices (as in a larger pattern it could have been matched to an element outside). With this procedure, the matching $M(p)$ has the following properties:

\begin{enumerate}
    \item The procedures to check whether a pattern $p$ is coherent and to construct $M(p)$ are effective.
\item $M(p)$ is a matching by paths of length less than $k$ between a subset of $\mathcal{O}(p)$ and a subset of $\mathcal{U}(p)$.
\item If $k \leq k'$ and $p \in \Lambda^{B_{(\Upsilon+1)k}}$, $p' \in \Lambda^{B_{(\Upsilon+1)k'}}$ are two coherent patterns such that $p'|_{B_{(\Upsilon+1)k}} =p$, then $M(p)\subset M(p')$. Thus we can define $M_k(y)$ for $y\in\Lambda^{F(s)}$ that does not contain non-coherent patterns as the union over $k \geq 1$ of $M(y|_{B_{(\Upsilon+1)k}})$.
\end{enumerate}

\begin{definition}
    A coherent pattern $p \in \Lambda^{B_{(\Upsilon+1)k}}$ is \define{matched} if for every pair $(g,h)\in M(p)$ we have $ p_{\Sigma,\pm}(g) = p_{\Sigma,\pm}(h)$.
\end{definition}

With this, we can define $Y_{\mathtt{GM}}\subset \Lambda^{F(S)}$ as the set of all configurations given by the following set of forbidden patterns:
\begin{enumerate}
    \item \textbf{Coherence rule:} For all $k \geq 1$, we forbid all patterns $p \in \Lambda^{B_{(\Upsilon+1)k}}$ which are not coherent.
    \item \textbf{Matching rule}: For all $k \geq 1$, we forbid all patterns $p \in \Lambda^{B_{(\Upsilon+1)k}}$ which are coherent but are not matched.
\end{enumerate}

It is clear from the observation above that $Y_{\mathtt{GM}}\subset \Lambda^{F(S)}$ is an effective subshift. Recall that we denote by $\psi\colon F(S) \to \G$ the canonical epimorphism and that for $y \in \Lambda^{F(S)}$ which is fixed by $\ker(\psi)$ we can define $\psi^*(y)\in \Lambda^{\G}$ given by $\psi^*(y)(\psi(g)) = y(g)$ for every $g \in F(S)$.

\begin{definition}
    The \define{generalized mirror shift} is the set of configurations $X_{\texttt{GM}}\subset \Lambda^{\G}$ given by \[ X_{\texttt{GM}} = \{\psi^*(y) :  y\in Y_{\mathtt{GM}} \mbox{ and } gy=y \mbox{ for every } g \in \operatorname{ker}(\psi) \}.     \]
\end{definition}

From~\Cref{prop_caractES} it follows that $X_{\texttt{GM}}$ is effectively closed by patterns. We shall give a brief description of what some configurations in $X_{\texttt{GM}}$ look like.

First, take $\xi \in \{0,1\}^{\G}$ as in~\Cref{lemma:distinct_nbhds}. This induces a configuration $\widehat{\xi}\in \{0,1\}^{F(S)}$ given by $\widehat{\xi}(g)=\xi(\psi(g))$. Notice that by definition of $\xi$, we have that $g,g'\in F(S)$ satisfy that $\widehat{\xi}(gt)=\widehat{\xi}(g't)$ for all $t \in B_{|g'g^{-1}|\Upsilon}$ if and only if $g'g^{-1}\in \ker(\psi)$. From the coherence rule it follows that if $y \in Y_{\texttt{GM}}$ is such that $y_{\{0,1\}}=\widehat{\xi}$, then the stabilizer of $y$ is precisely $\ker(\psi)$, and thus it induces a configuration $x \in X_{\texttt{GM}}$ such that $x_{\{0,1\}}=\xi$ and for all $g \in F(S)$, $x_{\Sigma}(\psi(g)) = y_{\Sigma}(g)$. Moreover, from the matching rule, the configuration $y$ carries a matching $M(y)\subset \mathcal{U}(y)\times \mathcal{O}(y)$ and satisfies that for every $(g,g')\in M(y)$, then $y_{\Sigma,\pm}(g)=y_{\Sigma,\pm}(g')$. Note that this induces naturally a matching $M(x) = \{ (\psi(g),\psi(g')) : (g,g')\in M(y)\}$ and thus for every $(g,g')\in M(x)$ we have $x_{\Sigma,\pm}(g)=x_{\Sigma,\pm}(g')$.

\subsection{Proof of~\Cref{mainthm:property_ET_groups_not_SS}} The result follows directly from the following theorem.

\begin{theorem}
    Let $\G$ be a finitely generated extraterrestrial group. The generalized mirror shift on $\G$ is effectively closed by patterns but not sofic.
\end{theorem}

\begin{proof}Let $\G$ be an extraterrestrial group. Take the generalized mirror shift $X_{\texttt{GM}}$ on $\G$ and suppose it is a factor of some $\G$-SFT $Y$. Without loss of generality, we may assume that $Y$ is $S$-nearest neighbor and that the factor map $\phi \colon Y \to X_{\texttt{GM}}$ is a $1$-block map.

We claim there exist two configurations $y,y'\in Y$ such that if we denote $x = \phi(y)$ and $x'=\phi(y')$ then

    \begin{enumerate}
        \item $x_{\Sigma,A} = x'_{\Sigma,A}$ and $x_{\{0,1\}} = x'_{\{0,1\}}$, that is, $x$ and $x'$ match on the $\{0,1\}$ layer and have the same symbols $\{\star,\mathfrak{u},\mathfrak{o}\}$, in particular $M(x)=M(x')$.
        \item There exists $(g_1,g_2)\in M(x)$ such that $x_{\Sigma,\pm}(g_1)=x_{\Sigma,\pm}(g_2)\neq x'_{\Sigma,\pm}(g_2) = x'_{\Sigma,\pm}(g_1)$.
        \item There exists a set $F\subset \G$ such that 
        \begin{enumerate}
        \item $y|_F = y'|_F$.
            \item $\mathcal{O}(x)$ and $\mathcal{U}(x)$ are contained in distinct connected components of $\Cay(\G,S)$ after removing $F$.
        \end{enumerate} 
    \end{enumerate}
Let us first show that the existence of such configurations $y,y'$ leads to a contradiction.

Let $W\subset \G$ be the union of the connected components of $\mathcal{U}(x)$ in $\Cay(\G,S)$ after removing $F$. This means that the $S$-neighborhood of any $g \in W$ is either in $W$ or in $F$. By property (3a) and the fact that $Y$ is $S$-nearest neighbor, we can construct a new configuration $y^*\in Y$ such that $y^*|_W = y|_W$ and $y^*|_{\G \setminus W}=y'|_{\G\setminus W}$. Let $x^* = \phi(y^*)$. As $\phi$ is a $1$-block map, it follows that $x^*|_W = x|_W$ and $x^*|_{\G \setminus W}=x'|_{\G \setminus W}$. By property (1) we have that $x^*_{\Sigma,A}=x_{\Sigma,A}$ and $x^*_{\{0,1\}}=x_{\{0,1\}}$, in particular we deduce that $M(x^*)=M(x)=M(x')$.

Let $(g_1,g_2)\in M(x^*)$ be the pair of elements from property (2) and note that by the previous property we have $g_1 \in \mathcal{U}(x^*)$ and $g_2 \in \mathcal{O}(x^*)$. On the one hand, as $x^*\in X_{\texttt{GM}}$, we have that $x^*_{\Sigma,\pm}(g_1)=x^*_{\Sigma,\pm}(g_2)$. On the other hand, by property (3b) we get that $x^*_{\Sigma,\pm}(g_1)=x_{\Sigma,\pm}(g_1)$ and $x^*_{\Sigma,\pm}(g_2)=x'_{\Sigma,\pm}(g_2)$. However by property (2) these two values are distinct, thus raising a contradiction.

Now we show that the required configurations $y,y'$ do exist. Let $\xi\in \{0,1\}^{\G}$ be a configuration which satisfies the hypothesis of~\Cref{lemma:distinct_nbhds}. Let $q$ be the size of the alphabet of $Y$ and fix an integer $m> 2\log_2(q)$. Since $\G$ is extraterrestrial, there exists $k \in \N$ such that for all $r \in \N$ there exists an $(m,k,r)$-UFO $(U_r,F_r,O_r)$.

For each $r \in \N$, let $W_r$ be the set of all $x \in X_{\mathtt{GM}}$ which satisfy the following properties

\begin{enumerate}
    \item $x_{\{0,1\}}=\xi$.
    \item For all $g \in \G\setminus (U_r\cup O_r)$, $x_{\Sigma,A}(g)=\star $ and $x_{\Sigma,\pm}(g) = +$.
    \item For all $g \in U_r$, $x_{\Sigma,A}(g)= \mathfrak{u}$.
    \item For all $g \in O_r$, $x_{\Sigma,A}(g)= \mathfrak{o}$.
\end{enumerate}

In simpler words, $W_r$ is the finite set of all configurations $x$ that have the fixed value $\xi$ on the second component (and thus there are no non-trivial identifications in $\G$), that satisfy $\mathcal{U}(x)=U_r$ and $\mathcal{O}(x)=O_r$, and that are constant equal to $(\star,+)$ everywhere else. Notice that all $x \in W_r$ have the same induced matching, which we will call $M_r$, and that thus the only difference between elements of $W_r$ are the symbols $\{-,+\}$ on $U_r\cup O_r$.

Let $M_{r,k}$ be the subset of $M_r$ which consists of pairs $(g,g')\in U_r\times O_r$ given by paths of length at most $k$. We claim that $2|M_{r,k}| \geq |U_r|$. Indeed, as $(U_r,F_r,O_r)$ is a UFO, it follows that there exists a matching $M'$ of $U_r$ to $O_r$ by paths of length at most $k$. If $2|M_{r,k}|<|U_r|$, there would exist $(g,g')\in M'$ such that neither $g$ nor $g'$ occurs in $M_{r,k}$, but this is not possible as $M_{r,k}$ is constructed by greedily matching vertices by lexicographic order. 

From the computation above, the fact that $|U_r|\geq m|F_r|$ and $m > 2\log_2(q)$ it follows that \[ 2^{|M_{r,k}|} \geq 2^{\frac{|U_r|}{2}} \geq 2^{\frac{m|F_r|}{2}} > q^{|F_r|}.    \]

From the inequality above, it follows that there exist $x^1,x^2 \in W_r$ and $y^1,y^2 \in Y$ with $\phi(y^1)=x^1$, $\phi(y^2)=x^2$ such that
\begin{enumerate}
    \item There is $(g,g')\in M_{r,k}$ such that $x^1_{\Sigma,\pm}(g)=x^1_{\Sigma,\pm}(g') \neq x^2_{\Sigma,\pm}(g') = x^2_{\Sigma,\pm}(g)$.
    \item $y^1|_{F_r}= y^2|_{F_r}$.
\end{enumerate} 

Define $y_r = g^{-1}y_1$ and $y'_r = g^{-1}y_2$. Take $(y,y')\in Y \times Y$ as any limit point of the sequence $(y_r,y'_r)_{r \in \N}$. We claim that $(y,y')$ satisfies the requirements.

First, by construction, for each $r$ we have that $\phi(y_r)_{\Sigma,A} = \phi(y'_r)_{\Sigma,A}$ and $\phi(y_r)_{\{0,1\}} = \phi(y'_r)_{\{0,1\}}$. By compactness it follows that $\phi(y)_{\Sigma,A} = \phi(y')_{\Sigma,A}$ and $\phi(y)_{\{0,1\}} = \phi(y')_{\{0,1\}}$.

Second, by construction we have that for each $r$ there is $h \in B_k$ such that $(1_G,h) \in M(\phi(y_r))$ and $\phi(y_r)_{\Sigma,\pm}(1_{\G}) = \phi(y_r)_{\Sigma,\pm}(h) \neq \phi(y'_r)_{\Sigma,\pm}(h) = \phi(y'_r)_{\Sigma,\pm}(1_{
\G})$. This property passes again by compactness to $(y,y')$.

Finally, Let $F\subset \G$ be the set of all positions where $y$ and $y'$ coincide and map to a $\star$, that is \[ F = \{g \in \G \setminus (\mathcal{U}(\phi(y)) \cup \mathcal{O}(\phi(y'))) : y(g)=y'(g)\}. \]

Suppose by contradiction that there is a path from a vertex in $(\mathcal{U}(\phi(y))$ to some vertex in $(\mathcal{O}(\phi(y))$ in $\Cay(\G,S)$ after removing $F$. Take $t\geq 0$ large enough such that said path is of length at most $t$ and such that it is contained in $B_t$. By compactness, for all $r$ large enough along a subsequence we have that $(y_r,y_{r'})$ coincide with $(y,y')$ on $B_t$, thus we get a path from a vertex in $(\mathcal{U}(\phi(y_r))$ to some vertex in $(\mathcal{O}(\phi(y_r))$ of length $t$ that avoids the elements in $g\in \G$ where $y_r(g)=y'_r(g')$. But by construction $y_r$ and $y'_r$ coincide in $gF_r$, thus we get a path from $gU_r$ to $gO_r$ of length at most $t$ that avoids $gF_r$. This contradicts the fact that $(U_r,F_r,O_r)$ is an $(m,k,r)$-UFO for any $r>t$.\end{proof}

The following consequence is immediate from the definitions in~\Cref{subsec_computability_shift}.

\begin{cor}
    Finitely generated extraterrestrial groups are not strongly self-simulable. Furthermore, if such a group is recursively presented, then it is not self-simulable.
\end{cor}

We note that in the case of groups that are not recursively presented, we do not obtain that extraterrestrial groups are not self-simulable, with the original definition of~\cite{BarSabSal2025s-selfsimulablegroups}. In fact, we shall construct an example of a (non-recursively presented) extraterrestrial group which is self-simulable.

To that end, we first remark that there exists a finitely generated amenable simple group whose word problem is not $\Pi_2^0$. This follows from work of Matui~\cite{Matui2006} where it is shown that the derived subgroup of the topological full group of a minimal $\Z$-subshift is finitely generated and simple, and furthermore, that uncountably many such groups can be obtained in this manner. Finally, Juschenko and Monod~\cite{JuschenkoMonod2013} showed that these groups are amenable. 

Furthermore, recall that a group action $\G\curvearrowright X$ is faithful if for every $g\in \G\setminus\{1_{\G}\}$ there is $x\in X$ such that $gx\neq x$. In~\cite[Proposition 2.10]{BarSabSal2025s-selfsimulablegroups} it is shown that any finitely generated group which admits a faithful computable action on a $\Pi_1^0$ subset of $\{0,1\}^{\N}$ has a $\Pi_2^0$ word problem. Hence (since $\G$ is simple) the only computable actions of $\G$ are given by the trivial map on a $\Pi_1^0$ set. 



\begin{example}
    Let $G$ be any finitely generated amenable simple group whose word problem is not $\Pi_2^0$. By~\cite[Proposition 2.10]{BarSabSal2025s-selfsimulablegroups}, it follows that the only computable actions of $G$ are given by the trivial map on a $\Pi_1^0$ set. 
    
    Consider $\Gamma = G \times G\times \Z$. By the argument above, the only computable actions of $\Gamma$ on $\Pi_1^0$ subsets are lifts of computable $\Z$-actions. By the main result of~\cite{Bar_2019_geometricsimulation} every such lift is the topological factor of a $\Gamma$-SFT. Hence it follows that $\Gamma$ is self-simulable and amenable. We shall later show in~\Cref{prop:AmenableET} that every amenable group is extraterrestrial.
\end{example}

\section{Examples of extraterrestrial groups}

The purpose of this section is to give examples of extraterrestrial groups and which thus satisfy the conclusion of~\Cref{mainthm:property_ET_groups_not_SS}

\subsection{Amenable and multi-ended groups}

In~\cite{AubBarSab_effectiveness_2017}  it is shown that amenable groups (Theorem 2.16 of \cite{AubBarSab_effectiveness_2017}) and multi-ended groups (Theorem 2.17 of \cite{AubBarSab_effectiveness_2017}) with decidable word problem are not self-simulable. We provide short alternative proofs showing that those groups are extraterrestrial, these proofs hold even if the word problem is undecidable. 

Recall that a discrete group $\G$ is amenable if and only if for every $\varepsilon>0$ and $K\Subset \G$ there is $F\Subset \G$ such that $|FK \setminus F| \leq \varepsilon|F|.$

\begin{proposition}
\label{prop:AmenableET}
Finitely generated infinite amenable groups are extraterrestrial.
\end{proposition}

\begin{proof}
Let $m\geq 1$ be a positive integer and fix a finite symmetric set of generators $S$ for $\G$ which contains the identity. As $\G$ is amenable, it follows that there exists $U\Subset \G$ such that $|US\setminus U| \leq \frac{1}{m}|U|$.

Take $F=US\setminus U$ and let $O\subset \G\setminus US$ be any set of size $|U|$ (this is possible since $\G \setminus US$ is infinite). We have that $|O|=|U|\geq m|F|$ and that $U$ and $O$ lie on distinct components of $\Cay(\G,S)$ after removing $F$. If we let $k=\max\{d_S(g,g'):g\in U\textrm{ and }g'\in O\}$, it follows that $(U,F,O)$ is an $(m,k,r)$-UFO for every $r \geq 1$ and thus $\G$ is extraterrestrial.
\end{proof}

Recall that the number of ends of a finitely generated group $\G$ is the limit as $n \to \infty$ of the number of infinite connected components that occur in $\Cay(\G,S)$ (for some set of generators $S$) after removing the ball of radius $n$. Alternatively, it is the number of equivalence classes of rays in $Cay(\G,S)$, where two rays $r_1$ and $r_2$ are said to be equivalent if there exists a ray that contains infinitely many vertices of both $r_1$ and $r_2$. It is well known that the number of ends of a group does not depend upon the choice of generators. We say that $\G$ is \define{multi-ended} if it has more than one end. Equivalently, $\G$ is multi-ended if for some finite set of generators $S$, there exists $F\Subset \G$ such that $\Cay(\G,S)$ has at least two infinite components after removing $F$.

\begin{proposition}
\label{prop:MultiEndedET}
Multi-ended groups are extraterrestrial.
\end{proposition}

\begin{proof}
If $\G$ is multi-ended, there exists a finite subset $F$ such that $\mathrm{Cay}(\G, S)$ has more than one infinite connected component after removing $F$. Given $m\in\N$, take $U$ in a connected component of $\mathrm{Cay}(\G, S)\setminus F$ such that $|U|\geq m|F|$ and $O$ in an other infinite connected component such that $|U|=|O|$. Take $k=\max\{d_S(g,g'):g\in U\textrm{ and }g'\in O\}$. It follows that there exists a matching between $U$ and $O$ by paths of length less than $k$. Moreover, since there is no path from $U$ to $O$ which avoids $F$, the distance between $U$ and $O$ in $\mathrm{Cay}(\G, S)\setminus F$ can not be $r$ for any $r\in \N$. We deduce that for any $r\in\N$, $(U,F,O)$ is a $(m,k,r)$-UFO and thus $\G$ is extraterrestrial.
\end{proof}

We now give a brief remark on Thompson's group $F$ (for an introduction to this group, see~\cite{Cannon_thompsongroups}). It is a well-known open problem to determine whether $F$ is amenable. In~\cite[Theorem 1.7]{BarSabSal2025s-selfsimulablegroups} it was shown that $F$ is nonamenable if and only if it is self-simulable. From this we deduce that amenability on this group is equivalent to it being extraterrestrial.

\begin{proposition}\label{prop:thompsonF}
    Thompson's group $F$ is amenable if and only if it is extraterrestrial.
\end{proposition}

\begin{proof}
    If $F$ were amenable, then by~\Cref{prop:AmenableET} it would be extraterrestrial. Conversely, as $F$ is finitely presented, if $F$ were extraterrestrial, then it would not be self-simulable by~\Cref{mainthm:property_ET_groups_not_SS}, and thus by~\cite[Theorem 1.7]{BarSabSal2025s-selfsimulablegroups} we would have that $F$ is amenable.
\end{proof}

We remark that from~\Cref{prop:MultiEndedET} it follows that many nonamenable groups are extraterrestrial, thus a priori~\Cref{prop:thompsonF} gives an ``easier'' condition to show amenability of $F$. We do not have any viable strategy to determine whether $F$ is extraterrestrial or not.

\subsection{Amenable extensions of multi-ended graphs}\label{subsec_amenable_cut}

In \cite{BarSabSal2025-tentacles}, we prove that the group $\Z\times \mathbb{F}_2$ is not self-simulable. It is an example of one ended non-amenable group which is not self-simulable. Using the language of Schreier graphs, we can state a generalization of this example.

\begin{theorem}
    \label{thm:AmenableExtension}
    Let $\G$ be a finitely generated group and $\H \leqslant \G$ an amenable subgroup. If for some symmetric generating set $S$ of $\G$, the Schreier graph $\Sch(\G,\H,S)$ is extraterrestrial, then so is $\G$.
\end{theorem}
Note that the choice of $S$ does not matter, as being extraterrestrial is quasi-isometry invariant. 
\begin{proof}
Let $\pi \colon \G \to \H \backslash \G$ be the canonical projection map given by $\pi(g)=\H g$. For $m \in \N$, define $\kappa(m)$ to be minimal such that for every $r \in \N$, there exists an $(m,\kappa(m),r)$-UFO in $\Sch(\G,\H,S)$. 

    Now, fix $m \in \N$. Let $(U_r,F_r,O_r)$ be a $(2m,\kappa(2m),r)$-UFO in $\Sch(\G,\H,S)$. Choose $U'_r \subset \G$ such that $\pi|_{U'_r}$ is a bijection onto $U_r$. Now, by the definition of distance in the Schreier graph, there exists $O'_r$ such that $\pi|_{O'_r}$ is a bijection onto $O_r$ and there is a complete matching with paths of length at most $\kappa(2m)$ between $U'_r$ and $O'_r$.

    Also choose $F'_r$ arbitrarily so that $\pi|_{F'_r}$ is a bijection onto $F_r$. Let now \[K_r = \{h \in \H : \mbox{ there exists } u \in U'_r \mbox{ and } f \in F'_r \mbox{ such that } d_S(u,hf) \leq r\}.\] In other words, $K_r$ is the set of heights of elements in the fibers of $F_r$ that are at distance at most $r$ from $U'_r$. Note that, as both $U'_r$ and $F'_r$ are finite, we also have that $K_r$ is finite.

    As $\H$ is amenable, there exists a finite set $T \Subset \H$ such that $|T K_r \setminus T | \leq |T|$. We define \[U = TU'_r,\quad O = TO'_r \mbox{ and } F = TK_rF'_r.\] By our previous definition, we have $|U| = |U'_r||T|=|U_r||T|$ and $|F| = |TK_r||F'_r| \leq 2 |T||F_r|$. Since $(U_r,F_r,O_r)$ is a $(2m,\kappa(2m),r)$-UFO, $|U_r| \geq 2m |F_r|$, and so $|U| \geq m|F|$.

    As by construction there is a complete matching $M'$ with paths of length at most $\kappa(2m)$ between $U'_r$ and $O'_r$, it follows that there is also a complete matching $M$ with paths of length at most $\kappa(2m)$ between $U$ and $O$. More explicitly, we may take $M = \{ (hu,ho) : (u,o) \in M', h \in T\}$.

    Finally, let $u = h u_0 \in U$ with $h \in T$ and $u_0 \in U'_r$ and consider a path $(hu_i)_{0\leq i\leq r}$ in $\Cay(\G,S)$ such that $hu_r \in O$. As any path in $\G$ projects onto a path in $\Sch(\G,\H,S)$ and since $(U_r,F_r,O_r)$ is a $(2m,\kappa(2m),r)$-UFO, there exists $i_0\in \{1,\dots,r\}$ for which $hu_{i_0} \in \H f$ for some $f \in F'_r$. But if $h u_{i_0} = h'f$ with $h' \in \H$, then $d_S(u_0,h^{-1}h'f) = i_0 \leq r$ and thus it must be that $h^{-1} h' \in K_r$. We conclude that $h u_{i_0} \in T K_r F'_r = F$. This means that any path of length at most $r$ between $U$ and $O$ must cross $F$, and so $(U,F,O)$ is a $(m,\kappa(2m),r)$-UFO.\end{proof}


\begin{cor}
    \label{cor:AmenableGroupExtension}
    Let $1 \to \H \to \G \to \K \to 1$ be a short exact sequence where $\H$ is amenable, $\G$ is finitely generated and $\K$ is extraterrestrial. Then $\G$ is extraterrestrial.
\end{cor}
\begin{proof}
    Let $\pi\colon \G\to \K$ the epimorphism from the sequence. For any finite generating set $S$ of $\G$, the Schreier graph $\Sch(\G,\H,S)$ is isomorphic to the Cayley graph $\Cay(\K,\pi(S))$. In particular, it is extraterrestrial. By~\Cref{thm:AmenableExtension}, $\G$ is extraterrestrial.
\end{proof}


\begin{cor}\label{cor:amenable_disconnections}
    Let $\G$ be a finitely generated group, and $\H \leq \G$ an amenable subgroup. If for some generating set $S$, the Schreier graph $\Sch(\G,\H,S)$ is multi-ended, then $\G$ is extraterrestrial.
\end{cor}

\begin{proof}
    This is an immediate consequence of~\Cref{thm:AmenableExtension} and~\Cref{prop:MultiEndedET}.
\end{proof}

We will say that a group admits an amenable cut if it satisfies the hypotheses of~\Cref{cor:amenable_disconnections}. All known examples of groups that are not self-simulable fall under this corollary. However, there are cases where it is more natural to construct UFOs geometrically, as we will see in the case of hyperbolic surface groups. As a particular case of this previous corollary, the result of \cite[Proposition 3.10]{BarSabSal2025s-selfsimulablegroups} which states that $\Z \times \mathbb{F}_2$ is not self-simulable falls into the following corollary.

\begin{cor}
\label{prop:AxME}
Let $\G = \H \times \K$ such that $\H$ is amenable and $\K$ is not one-ended.  The group $\G$ is extraterrestrial.
\end{cor}
\begin{proof}
    Consider the short exact sequence $1 \to \H \to \G \to^\pi \K \to 1$ where $\pi$ is the natural projection to the second coordinate. If $\K$ is zero-ended, then it is finite and so $\G$ is amenable. By~\Cref{prop:AmenableET}, it is extraterrestrial.  Otherwise, $\K$ is multi-ended and so by~\Cref{prop:MultiEndedET} it is extraterrestrial. We conclude with ~\Cref{cor:AmenableGroupExtension}.
\end{proof}

\subsection{Amalgamated free products and HNN extensions}

In this section we will provide natural conditions under which amalgamated free products and HNN extensions are extraterrestrial. We shall only provide the relevant definitions here and refer the reader to~\cite{Lyndon2001} for a more gentle introduction.

Given groups $\G_1,\G_2,\H$ and embeddings $\phi_1\colon \H\to \G_1$, $\phi_2\colon \H \to \G_2$ the \define{amalgamated free product} of $\G_1$ and $\G_2$ along $\H$ is given by \[  \G_1 \ast_{\H}\G_2 = \langle \G_1,\G_2\; |\; \phi_1(h)=\phi_2(h) \mbox{ for all } h \in \H\rangle.  \]

A normal form for an amalgamated free product can be obtained as follows. For $i \in \{1,2\}$, choose a set of representatives $T_i$ of the cosets $\phi_i(\H) \backslash \G_i$ which contains the identity. A normal form is a sequence $x_0,\dots,x_n$ such that $x_0 \in \phi_1(\H)$ and for $j \geq 1$ each $x_j$ is either in $T_1\setminus\{1\} $ or $T_2\setminus\{1\}$ in an alternating way. Every element is equal to a normal form, and distinct normal forms correspond to distinct elements, see~\cite{Lyndon2001}.

\begin{proposition}\label{prop:amalgamated_free_ET}
    Let $\G_1,\G_2,\H$ be finitely generated groups with $\H$ amenable. Let $\phi_1 \colon \H \to \G_1$ and $\phi_2 \colon \H \to \G_2$ be proper embeddings. Then the amalgamated free product $\G_1 \ast_{\H}\G_2$ is extraterrestrial. 
\end{proposition}

\begin{proof}
    Take $S_1,S_2$ finite symmetric set of generators of $\G_1,\G_2$ respectively and take $S=S_1\cup S_2$. It is clear that $S$ generates $\G_1 \ast_{\H}\G_2$. 
    For each $i \in \{1,2\}$ choose a set of representatives $T_i$ of the cosets $\phi_i(\H) \backslash \G_i$ which contains the identity. As the embeddings are proper, we have that $|T_i|\geq 2$.
    
    By~\Cref{cor:amenable_disconnections} it suffices to show that the Schreier graph $\Sch(\G_1 \ast_{\H}\G_2, \H,S)$ is multi-ended. Take $g_1 \in T_1\setminus\{1\}$ and $g_2 \in T_2\setminus\{1\}$. From the structure of the normal forms, it is easy to show that for any $n \in \N, \phi_1(\H) (g_1 g_2)^n$ and $\phi_1(\H) (g_2 g_1)^n$ lie in distinct connected components of $\Sch(\G_1 \ast_{\H}\G_,\phi_1(\H),S)$ after removing $\phi_1(\H)$. Thus the rays $(\phi_1(\H) (g_1 g_2)^n)_{n \in \N}$ and $(\phi_1(\H) (g_2 g_1)^n)_{n \in \N}$ are non equivalent in $\Sch(\G_1 \ast_{\H}\G_2,\H,S)$, and it is hence multi-ended.\end{proof}

We remark that~\Cref{prop:amalgamated_free_ET} may fail if one of the embeddings is not proper. For instance, if we take $\G_1=\H=\Z$, $\phi_1$ the identity, $\G_2 = F_2\times F_2$ and $\phi_2\colon \Z \to F_2 \times F_2$ any embedding, then $\G_1 \ast_{\H}\G_
2 \cong F_2\times F_2$ which is not extraterrestrial (due to~\Cref{mainthm:property_ET_groups_not_SS} and the fact that it is self-simulable).

Next we consider the case of HNN-extensions. Let $\G$ be a group, $\H_1,\H_2$ two subgroups of $\G$ and $\phi\colon \H_1 \to \H_2$ an isomorphism. The \define{HNN-extension} of $\G$ relative to $\phi$ is the group \[ \G^{\ast\phi} = \langle \G, t \;\mid\; t^{-1} h t = \phi(h) \mbox{ for all } h \in \H_1\rangle.  \]
In the presentation above we call $t$ the \define{stable generator}. We shall need the following lemma of Britton whose proof can be found for instance in~\cite{Lyndon2001}.

\begin{lemma}[Britton's lemma] Let $n \geq 1$ and $g = g_0t^{\epsilon_1}g_1\dots t^{\epsilon_n}g_n \in \G^{\ast \phi}$ where each $g_i \in \G$ and $\epsilon_i \in 
\{-1,+1\}$. If $g=1_{\G^{\ast \phi}}$, then there exists a subword of the form $t^{-1}g_it$ with $g_i \in \H_1$ or $tg_jt^{-1}$ with $g_j \in \H_2$. 
\end{lemma}

This can be refined to a normal form (with unique representations for elements), but this lemma suffices for our purposes.

We obtain the following proposition.

\begin{proposition}\label{prop:HNN_ET}
    Let $\G$ be a finitely generated group and $\phi\colon \H_1 \to \H_2$ an isomorphism between amenable subgroups of $\G$. Then  $\G^{\ast\phi}$ is extraterrestrial.
\end{proposition}

\begin{proof}
    Let $t$ be the stable generator of $\G^{\ast \phi}$, $F$ a finite generating set of $\G$ and let $S = F\cup \{t,t^{-1}\}$. Recall that for some word $w$ on the generators we write $\underline{w}$ for the corresponding element in the group.
    
    Consider the Schreier graph $\Sch(\G^{\ast\phi},\H_1,S)$. Notice that if we remove the vertex $\H_1$, the rays $\{\H_1 t^n\}_{n \geq 1}$ and $\{\H_1 t^{-n}\}_{n \geq 1}$ lie in the same connected component as $\H_1 t$ and $\H_1 t^{-1}$ respectively, thus it suffices to show that $\H_1 t$ and $\H_1 t^{-1}$ lie in distinct connected components.

    Let $x_1\dots x_n \in S^*$ be such that $\H_1 t x_1\dots x_n = \H_1 t^{-1}$. We will show that there is a prefix $u_0$ of $tx_1\dots x_n$ such that $\underline{u_0} \in \H_1$. Notice that there is $a \in \H_1$ such that \[ t x_1\dots x_n ta = 1_{\G^{\ast \phi}}.  \]

    Let $w_0 = t x_1\dots x_n ta$. By Britton's lemma, either there exists a subword of $w_0$ of the form $t^{\epsilon}xt^{-\epsilon}$ with $x \in F^*$ such that either $\epsilon=-1$ and $\underline{x}\in \H_1$, or $\epsilon=1$ and $\underline{x} \in \H_2$. Notice that we may replace any of these subwords by some word in $F^*$ and obtain another word which also represents $1_{\G^{\ast \phi}}$ and has one less occurrence of both $t$ and $t^{-1}$. Iteratively define $w_{i+1}$ from $w_i$ by doing the leftmost possible replacement as above.

    As the number of pairs of $t,t^{-1}$ is reduced with each iteration, it follows that there is some $k \geq 0$ for which there is a prefix $u_k$ of $w_k$ that has the form $u_k=txt^{-1}$ for some $x\in F^*$ with $\underline{x} \in \H_2$. In particular, $\underline{u_k}=\underline{txt^{-1}}\in \H_1$. Write $w_k = u_kv_k$.

    Next we are going to define prefixes $u_i$ of $w_i$ for every $i \in \{0,\dots,k\}$ which satisfy two properties: that $\underline{u_i}\in \mathbb{H}_1$ and that $u_i$ ends with $t^{-1}$. To this end, suppose we have defined $w_{i+1}=u_{i+1}v_{i+1}$ and wish to define $u_i$. Notice that if we undo the procedure described above to obtain $w_{i+1}$ from $w_i$, we are replacing a word in $F^*$ by a fragment of the form $t^{\epsilon}xt^{-\epsilon}$ with $x\in F^*$, hence the word we are going to replace is either completely contained in $u_{i+1}$ or $v_{i+1}$. If the replacement if performed in $v_{i+1}$, define $u_{i}=u_{i+1}$; otherwise, let $u_{i}$ be the word obtained from $u_{i+1}$ by performing this replacement and notice that $u_i$ still ends by $t^{-1}$. Notice that in both cases we have that $\underline{u_i}\in \mathbb{H}_1$ and $u_i$ ends with $t^{-1}$ as required.

    After this procedure, we obtain a prefix $u_0$ of $w_0$ which ends in $t^{-1}$. As $w_0$ ends by $ta$, it follows that $u_0$ is a prefix of $tx_1\dots x_n$. Furthermore, by construction we have $\underline{u_0}\in \mathbb{H}_1$, hence we deduce that $\H_1u_0 = \H_1$. 

    From the argument above we deduce that any path from $\H_1t$ to $\H_1 t^{-1}$ must cross $\H_1$.
 Hence $\Sch(\G^{\ast\phi},\H_1,S)$ is multi-ended, and thus by~\Cref{cor:amenable_disconnections}, $\G$ is extraterrestrial.\end{proof}

In what follows we will use the language of Bass-Serre theory. As before, we will provide only the essential definitions and refer the reader to~\cite{Serre1980} for a proper introduction.

\begin{example}\label{ex:BS}
    Given $m,n\in \Z\setminus \{0\}$, the Baumslag-Solitar group $\BS(m,n)$ is the group given by the presentation \[ \BS(m,n) = \langle a,b \;\mid\; ba^mb^{-1}=a^n\rangle. \]
    $\BS(m,n)$ can be seen as the HNN-extension of $\Z = \langle a\rangle$ relative to the isomorphism $\phi\colon \langle a^m \rangle \to \langle a^n \rangle$ that sends $a^m$ to $a^n$. This group acts cocompactly by automorphisms without inversions and with infinite cyclic edge and vertex stabilizers on its Bass-Serre tree. In~\Cref{fig:BStree} we show the Bass-Serre tree in the case $m=1, n=2$. 
\end{example}

\begin{figure}[ht!]
    \centering
    \begin{tikzpicture}
        \draw[thick, gray] (2,0) to (3,0.33);
         \draw[thick, gray] (2,0) to (3,-0.33);
        \draw[thick] (2,0) to (0,-1);
        \draw[thick] (0,-1) to (-2,-2);
        \draw[thick] (0,-1) to (2,-2);
        \draw[thick] (-2,-2) to (-3,-3);
        \draw[thick] (-2,-2) to (-1,-3);
        \draw[thick] (2,-2) to (1,-3);
        \draw[thick] (2,-2) to (3,-3);
        \foreach \j in {-3,-1,1,3}{
        \draw[thick, gray] (\j,-3) to (\j-0.5,-3.5);
        \draw[thick, gray] (\j,-3) to (\j+0.5,-3.5);
        }
        \draw[fill = black] (2,0) circle (0.1);
        \draw[fill = black] (0,-1) circle (0.1);
        \draw[fill = black] (-2,-2) circle (0.1);
        \draw[fill = black] (2,-2) circle (0.1);
        \draw[fill = black] (-3,-3) circle (0.1);
        \draw[fill = black] (-1,-3) circle (0.1);
        \draw[fill = black] (1,-3) circle (0.1);
        \draw[fill = black] (3,-3) circle (0.1);
        \node at (2-0.5,0+0.3) {$b^{-1}\langle a\rangle$};
        \node at (0-0.4,-1+0.2) {$\langle a\rangle$};
        \node at (-2-0.5,-2) {$b\langle a\rangle$};
        \node at (2-0.6,-2) {$ab\langle a\rangle$};
        \node at (-3-0.6,-3) {$b^2\langle a\rangle$};
        \node at (-1-0.7,-3) {$bab\langle a\rangle$};
        \node at (1-0.7,-3) {$ab^2\langle a\rangle$};
        \node at (3-0.7,-3) {$abab\langle a\rangle$};
        
    \end{tikzpicture}
    \caption{A finite portion of the Bass-Serre tree of $\BS(1,2)$}
    \label{fig:BStree}
\end{figure}
From the cases of amalgamated free products and HNN-extensions, we can deduce a general proposition about groups acting on trees. For this, we introduce the notion of graph of groups. A \define{graph of groups} is the data of 
\begin{enumerate}
    \item An oriented graph $G = (V,E)$ with origin and end maps $o,t \colon E \to V$ and inversion map $i \colon E \to E$.

    \item A family of \define{vertex groups} $(\G_v)_{v \in V}$.
    \item A family of \define{edge groups} $(\H_e)_{e \in E}$ with monomorphisms $\psi_{e,o} \colon \H_e \to \G_{o(e)}$ and $\psi_{e,t} \colon \H_e \to \G_{t(e)}$.
\end{enumerate}
Then, if $T$ is a covering tree of $G$, the \define{fundamental group} of the graph of groups is the group generated by the vertex groups $(\G_v)_{v \in V}$ and additional generators $s_e$ for every $e\in E$ subjected to the following relations. 
\begin{enumerate}
    \item $\forall e \in E, s_{i(e)} = s_e^{-1}.$
    \item $\forall e \in T, s_e = 1.$
    \item $\forall e \in E, \forall x \in \H_e, s_e^{-1}\psi_{e,o}(x) s_e = \psi_{e,t}(x)$.
\end{enumerate}
The choice of a covering tree does not change the fundamental group up to isomorphism \cite[\textsection 5, Proposition 20]{Serre1980}. 
\begin{example}
    The fundamental group of a segment of groups \tikz[baseline=-.7ex]{
		\node[label={$\G_1$},shape=circle,draw=black] (A) at (0,0) {};
		\node[label={$\G_2$},shape=circle,draw=black] (B) at (1,0) {};
		\path (A) edge[->] node [above] {$\H$} (B);
	} is the free product with amalgamation $\G_1 \ast_\H \G_2$. The fundamental group of a loop of groups \tikz[baseline=-.7ex]{
		\node[label=$\G$,shape=circle,draw=black] (A) at (0,0) {};
		\path (A) edge [loop right] node [fill=white] {$\H$} (A);
	} is the HNN-extension $\G^{\ast \psi_{e,t}\circ \psi_{e,o}^{-1}}$.
\end{example}

Fundamental groups of graphs of groups hence are a generalization of amalgamated free products and HNN-extensions, moreover, they are built inductively from these two constructions.

If $\G$ acts cocompactly by automorphisms without edge inversions on a tree $T$, then there exists a surjective graph morphism $\pi:T \to T/\G$, and $T/\G$ is finite. Since $\G$ acts without edge inversions, an element that stabilizes an edge also stabilizes its ends. Hence, for every $e\in E(T)$, there are monomorphisms $\psi_{e,o}$ and $\psi_{e,t}$ from $\Stab(e)$ to the stabilizers of its ends. We can then endow $G = T/\G$ with a graph of groups structure by setting $G_{\pi(v)} = \Stab(v)$ for every $v\in V(T)$ and $G_{\pi(e)} = \Stab(e)$ for every $e \in E(T)$ with the monomorphisms induced by the $\psi_{e,o},\psi_{e,t}$. The structure theorem of groups acting on trees ~\cite[\textsection 5, Theorem 13]{Serre1980} states that the graph of groups thus defined has $\G$ as a fundamental group.

\begin{proposition}\label{prop:action_on_tree}
Let $\G$ be a finitely generated group. Suppose that $\G$ acts cocompactly  by automorphisms without edge inversions on a tree. Suppose further that no vertex is fixed by $\G$ and that edge stabilizers are amenable. Then $\G$ is extraterrestrial.
\end{proposition}

\begin{proof}
    By the structure theorem of groups acting on trees, $\G$ may be realized as the fundamental group of a finite graph of groups with amenable edges. As no vertex is fixed by $\G$, we can also observe that $\G$ is not a vertex group of this graph of groups. But since fundamental groups of graphs of groups are built inductively by amalgamated free products and HNN-extensions, it follows that we may assume that either $\G$ is an HNN-extension over an amenable subgroup, or it is an amalgamated free product over an amenable subgroup. In this case, since $\G$ is not a vertex group of the graph of groups, this amenable subgroup must be a proper subgroup. We conclude with~\Cref{prop:amalgamated_free_ET} and~\Cref{prop:HNN_ET}.
\end{proof}

We remark that since edge stabilizers inject in vertex stabilizers, it follows from this result that if the infinite group $\G$ acts co-compactly by automorphisms without edge inversions and with amenable \emph{vertex} stabilizers on a tree, then $\G$ is extraterrestrial. Indeed, if a vertex is universally fixed, then $\G$ is infinite and amenable and we conclude by \Cref{prop:AmenableET}, and if not, then $\G$ verifies the condition of \Cref{prop:action_on_tree}. 

We also note that that fundamental groups of graphs of groups with at least one amenable edge that embeds properly, are extraterrestrial. This occurs because either the amenable edge disconnects the graph, in which case the group can be written as a free product with amalgamation over this edge, or it doesn’t, in which case there is a covering tree which doesn’t contain this edge. Contracting along this tree, we obtain a graph with a single vertex and multiple loops, one of which is the amenable edge, and so the group can be described as an HNN extension over this edge.

A \emph{generalized Baumslag-Solitar group} is a group that has a faithful cocompact action on a tree by automorphisms, without edge inversions and with infinite cyclic edge and vertex stabilizers. An immediate corollary of~\Cref{prop:action_on_tree} is the following.

\begin{cor}\label{cor:BS_has_ET}
Generalized Baumslag-Solitar groups are extraterrestrial.
\end{cor}

We end this section with an application of~\Cref{prop:action_on_tree}. A theorem of Paulin~\cite{Paulin_91} building on Rips theory shows that every one-ended hyperbolic group whose outer automorphism group is infinite splits as a graph of groups over virtually cyclic subgroups. See also~\cite[Theorem 1.2]{GuiLev15} for a modern reference in the context of relatively hyperbolic groups. We obtain the following corollary as an immediate consequence of~\Cref{prop:action_on_tree}.

\begin{cor}\label{cor:hyp_outer_ext}
    Every hyperbolic group $\G$ with an infinite outer automorphism group is extraterrestrial.
\end{cor}

We remark that examples of hyperbolic groups with infinite outer automorphism group include free groups $F_n$ for $n\geq 2$, fundamental groups of closed surfaces (we shall give two distinct proofs later on for these groups) and free products of (non-trivial) hyperbolic groups~\cite{CulVog}. In contrast, we note that this property does not hold for fundamental groups of closed hyperbolic 3-manifolds due to Mostow rigidity~\cite{Mostow}.

\subsection{Groups quasi-isometric to the hyperbolic plane}

In this section we show that any finitely generated group which is quasi-isometric to the hyperbolic plane is extraterrestrial. We will first provide a direct proof.

\begin{definition}\label{def:pentagonmodel}
The \emph{pentagon model} is the graph with vertices $\Z^2$, and edges
\[ \{\{ (m, n), (m+1, n) \} : m, n \in \Z \} \cup \{ \{ (m, n), (2m, n-1) \} : m, n \in \Z\}. \]
\end{definition}

\begin{figure}[ht!]
    \centering
    \begin{tikzpicture}[scale = 0.8]
   \draw[thick, gray] (-8,0) -- (-8,+1);
   \draw[thick, gray] (0,0) -- (0,+1);
   \draw[thick, gray] (8,0) -- (8,+1);
   \draw[thick, gray] (-8,0) -- (-10,0);
   \draw[thick, gray] (-8,-2) -- (-9,-2);
   \draw[thick, gray] (-8,-3.5) -- (-8.5,-3.5);
   \draw[thick, gray] (-8,-4.25) -- (-8.3,-4.25);
   \draw[thick, gray] (8,0) -- (10,0);
   \draw[thick, gray] (8,-2) -- (9,-2);
   \draw[thick, gray] (8,-3.5) -- (8.5,-3.5);
   \draw[thick, gray] (8,-4.25) -- (8.3,-4.25);
    \foreach \i in {-8,-4,0,4}{
        \draw[thick] (\i,0) rectangle (\i+4,-2);
        \draw[fill = black] (\i,0) circle (0.1);
        \draw[fill = black] (\i+4,0) circle (0.1);
        }
    \foreach \i in {-8,-6,-4,-2,0,2,4,6}{
        \draw[thick] (\i,-2) rectangle (\i+2,-3.5); 
        \draw[fill = black] (\i,-2) circle (0.1);
        \draw[fill = black] (\i+2,-2) circle (0.1);
    }
    \foreach \i in {-8,...,7}{
        \draw[thick] (\i,-3.5) rectangle (\i+1,-4.25); 
        \draw[fill = black] (\i,-3.5) circle (0.1);
        \draw[fill = black] (\i+1,-3.5) circle (0.1);
    }
    \foreach \i in {-16,...,16}{
        \draw[thick, gray] (\i/2,-4.25) -- (\i/2,-4.6); 
        \draw[fill = black] (\i/2,-4.25) circle (0.1);
    }
    \end{tikzpicture}
    \caption{A finite portion of the pentagon model.}
    \label{fig:Pentagon_model}
\end{figure}

In~\Cref{fig:Pentagon_model} a representation of the pentagon model is shown. It is well known that the pentagon model is quasi-isometric to the hyperbolic plane $\mathcal{H}= \{z \in \C : \operatorname{Im}(z)>0\}$ with the usual hyperbolic metric.  An explicit quasi-isometry from the vertices of the pentagon model to $\mathcal{H}$ is given by \[ (m,n) \mapsto m2^n+2^ni.  \]

\begin{lemma}\label{lem:pentagon_is_ET}
    The pentagon model is extraterrestrial.
\end{lemma}

\begin{proof}
 Let $m,r \in \N$ be arbitrary. An $(m,6m+1,r)$-UFO is given by 
 \begin{align*}
     U & = \{ (h,v) \in \Z^2 :  -3m \leq h < 0 \mbox{ and } 0\leq v < r\},\\
     F & = \{ (0,v) \in \Z^2 :  -r \leq v \leq 2r-1\},\\
     O & = \{ (h,v) \in \Z^2 :  0 < h \leq 3m \mbox{ and } 0\leq v < r\}.
 \end{align*}
Indeed, we have $|U|= 3mr$ and $|F| = 3r$, thus $|U|\geq m|F|$. A matching by paths of length at most $6m+1$ is given by $M = \{\bigl((-h,v),(h,v)\bigr)\in U \times O : 0 < h \leq 3m \mbox{ and } 0\leq v < r\}$, the paths are given by just increasing the first coordinate one by one. Finally, notice that any path from $U$ to $O$ must necessarily contain a vertex with horizontal coordinate $0$. If said path avoids $F$, then this vertex necessarily has vertical coordinate either smaller than $-r$ or larger than $m+r$. As each edge in this graph modifies the vertical coordinate at most by one, said path has length at least $r$.
\end{proof}

\begin{proposition}
    Let $\G$ be a finitely generated group which is quasi-isometric to the hyperbolic plane $\mathcal{H}$. Then $\G$ is extraterrestrial.
\end{proposition}

\begin{proof}
    Let $S$ be a generating set for $\G$. As both the pentagon model and $\Cay(\G,S)$ are quasi-isometric to $\mathcal{H}$, it follows that the pentagon model is quasi-isometric to $\Cay(\G,S)$. By~\Cref{mainthm:property_ET_QI_invariant} and~\Cref{lem:pentagon_is_ET} we deduce that $\Cay(\G,S)$ is extraterrestrial.
\end{proof}

A Fuchsian group is a discrete subgroup of $\mathrm{PSL}(2,\R)$. These groups act faithfully on the hyperbolic plane by M\"obius transformations. In particular, when they act cocompactly, they give a natural class of groups which are finitely generated and quasi-isometric (by the Schwarz-Milnor lemma, see~\cite[Proposition 8.19]{Bridson1999}) to $\mathcal{H}$. We obtain the following corollary

\begin{cor}\label{cor:Fuchsian_ET}
    Every cocompact Fuchsian group is extraterrestrial.
\end{cor}

Another interesting class of groups which we can completely classify are fundamental groups of surfaces. In the case of surfaces which are not closed (that is, either they are not compact or they have boundary), the fundamental group is free and thus is extraterrestrial by~\Cref{prop:MultiEndedET}. For closed surfaces we may discard the case of non-orientable surfaces, as they admit an orientable 2-cover and thus their fundamental groups are commensurable to the fundamental group of an orientable surface. Therefore, the only remaining case is when the surface is closed and orientable and the only remaining parameter is the genus $g$. If $g \leq 1$ the fundamental group is amenable and thus is extraterrestrial by~\Cref{prop:AmenableET}. Finally, if $g\geq 2$ the fundamental group is given by \[ \pi_1(\Sigma_g)=\langle a_1,b_1,...,a_g,b_g\;\mid\;\;[a_1,b_1]\ldots [a_g,b_g]=1 \rangle.  \]
See for instance~\cite[Chapter 4]{massey1977algebraic}. These groups are all commensurable for $g \geq 2$ and quasi-isometric to the hyperbolic plane, thus by~\Cref{lem:pentagon_is_ET} we obtain the following corollary.

\begin{cor}\label{cor:surfaces}
    The fundamental group of any 2-manifold is extraterrestrial.
\end{cor}

\subsection{An amenable cut in the surface group}

We now give an alternative proof that the fundamental group of a closed orientable surface of genus $g = 2$ (from now on, the surface group) is extraterrestrial. This has the advantage of showing explicitly that this group admits an amenable cut in the sense of~\Cref{cor:amenable_disconnections}. Before giving a proof, let us outline how one can ``see this immediately'' in the previous proof. Surface groups are hyperbolic, and in a hyperbolic group $\G$, every element of infinite order will define a quasi-geodesic, i.e. $n \mapsto g^n$ is a quasi-isometric embedding of $\mathbb{Z}$ into $\G$, see~\cite[Corollary 3.10]{Bridson1999}.

Thus, if we pick any torsion-free element $g$, then $\langle g \rangle$ induces a quasi-geodesic in the pentagon model. It is easy to convince oneself that such a quasi-geodesic must roughly cut the space into two parts. Now one just has to find elements of the group such that right translation by them moves one respectively to the left and right side of the path. It is easy to convince oneself that these exist (one may also construct this explicitly using the Cayley graph depicted in~\cite{AubBarMou2019_DPsurface} and perform an analogous argument directly there).

For the precise proof, we essentially directly obtain the statement from the literature, though it takes a bit of work to translate everything to our setting. Specifically, we will use a theorem from \cite{scott1977endspairs}. Thus, we need to recall the definition of ends used in \cite{scott1977endspairs}, as it is slightly different from ours. Let $\G$ be a group and $C$ a subgroup. Let $C\backslash \G$ denote the quotient of $\G$ by the left action of $C$. Then $\G$ acts on the right on $C\backslash \G$. Define $\overline{\Z_2}(C\backslash \G)$ as the $\Z_2$-vector space of subsets of $C\backslash \G$ under symmetric difference, and $\Z_2(C\backslash \G)$ as the subspace of vectors with finite support (i.e.\ the support consists of finitely many cosets). Define now $E(C\backslash \G) = \overline{\Z_2}(C\backslash \G)/\Z_2(C\backslash \G)$ and let $C^0(\G, E(C\backslash \G)) = E(C\backslash \G)$ and $C^1(\G, E(C\backslash \G)) = \{\phi : \G \to E(C \backslash \G)\}$ be the usual cochain groups. Take the canonical map $\partial_1 : C^0(\G, E(C\backslash \G)) \to C^1(\G, E(C\backslash \G))$ where for $\phi \in E(\G,C)$ and $g \in \G$ we have \[\partial_1(\phi)(g) = g\phi \cdot \phi.\]

Set the $0$-th cohomology group as $H^0(\G,E(C\backslash \G))= \operatorname{ker}(\partial_1).$

\begin{definition}
$e(\G, C)$ is the dimension, as a $\Z_2$-vector space, of the group $H^0(\G, E(C\backslash \G))$.
\end{definition}

\begin{lemma}
$e(\G, C)$ is the number of ends of the Schreier graph of $\G$ with respect to $C$.
\end{lemma}

\begin{proof}
The elements $\phi \in H^0(\G, E(C\backslash \G))$ are precisely the elements of $E(C\backslash \G)$ which satisfy that for all $g \in \G$ then $g\phi \cdot \phi = \phi$. That is, it is the vector space of subsets of $C\backslash \G$ (modulo finite differences) which are invariant under right translation (modulo finite differences). 

Take some finite symmetric generating set $S$ of $\G$ and consider the Schreier graph $\Gamma = \Sch(\G,C,S)$. Thus in terms of the Schreier graph, $e(\G,C)$ is the dimension of the vector space of subsets of the Schreier graph (modulo finite differences) which are invariant under right translation (modulo finite differences).

Suppose there are at least $k$ ends in $\Sch(\G,C,S)$, then we can find a finite subset of the Schreier graph $D \subset C\backslash \G$ such that $\Sch(\G,C,S)$ has at least $k$ infinite connected components after removing $D$. The characteristic functions of these components are elements of $H^0(\G, E(C\backslash \G))$ which are linearly independent, thus $e(\G,C) \geq k$. On the other hand, suppose that the dimension is at least $k$, and let $v_1, \ldots, v_k \in H^0(\G, E(C\backslash \G))$ be linearly independent. We can find a finite subset $D$ of $C\backslash \G$ outside which the $v_i$ are closed under $S$-translations (i.e, for $i\in \{1,\dots,k$ and $s \in S$ take $D_{i,s}$ as the finite set of elements where $v_i$ differs from $sv_i$ and then take $D$ as the union of the $D_{i,s}$). Thus, their supports must be formed out of connected components of $\Sch(\G,C,S)$ after removing $D$. For the dimension to be at least $k$, we need at least $k$ connected components.
\end{proof}

\begin{proposition}
\label{thm:SurfaceCut}
The surface group has an amenable cut.
\end{proposition}

\begin{proof}
A subsurface $X$ of a closed surface $F$ is called~\emph{incompressible} if no boundary component of $X$ bounds a disc. Lemma 2.2 in \cite{scott1977endspairs} states that if $\G$ is the fundamental group of a closed surface $F$ and $\H$ is the fundamental group of a compact, incompressible subsurface $X$ of $F$, then $e(\G, \H)$ equals the number of boundary components of $X$. Recall that $e(\G, \H)$ is just the number of ends of the Schreier graph of $\H$, so it suffices to find an incompressible subsurface $X$ in the genus $2$ closed surface $F$, such that $X$ has two boundary components, and the fundamental group of $X$ is amenable.

Such $X$ is shown in~\Cref{fig:cut}: there are clearly two boundary components, and $\pi_1(X) = \Z$.
\end{proof}

\begin{figure}[ht!]
    \centering
    \begin{tikzpicture}[scale = 1.5]
		\draw[thick, fill = black!10] (0,0) ellipse (1.8 and .9);
		\draw[thick, fill = black!10] (3.5,0) ellipse (1.8 and .9);
		\draw[fill = black!10, black!10] (0,-0.6) rectangle (3,0.6);
        \draw[thick] (1.3,0.6)--(2.2,0.6);
        \draw[thick] (1.3,-0.6)--(2.2,-0.6);
		\begin{scope}[scale=.8, shift={(-0.2,0)}]
		\path[rounded corners=24pt, fill = white] (-.9,0)--(0,.6)--(.9,0) (-.9,0)--(0,-.56)--(.9,0);
		\draw[rounded corners=28pt] (-1.1,.1)--(0,-.6)--(1.1,.1);
		\draw[rounded corners=24pt] (-.9,0)--(0,.6)--(.9,0);
		\end{scope}
		\begin{scope}[scale=.8, shift={(4.7,0)}]
		\path[rounded corners=24pt, fill = white] (-.9,0)--(0,.6)--(.9,0) (-.9,0)--(0,-.56)--(.9,0);
		\draw[rounded corners=28pt] (-1.1,.1)--(0,-.6)--(1.1,.1);
		\draw[rounded corners=24pt] (-.9,0)--(0,.6)--(.9,0);
		\end{scope}
        \begin{scope}[shift={(0.3,0)}]
        \draw[thick, dashed, fill = red!40, opacity = 0.2] (1.3,0.6) -- (1.8,0.6) .. controls (2,0) .. (1.8,-0.6) -- (1.3,-0.6) .. controls (1.5,0) .. cycle;
        \end{scope}
        \draw[thick, fill = red!40, opacity = 0.5] (1.3,0.6) -- (1.8,0.6) .. controls (1.6,0) .. (1.8,-0.6) -- (1.3,-0.6) .. controls (1.1,0) .. cycle;
        \node at (1.4,0) {$X$};
        \node at (-1,-0.5) {$F$};
		\end{tikzpicture}
    \caption{The cut $X$ used in~\Cref{thm:SurfaceCut}. }
    \label{fig:cut}
\end{figure}

\bibliographystyle{abbrv}
\bibliography{alexandria}

\Addresses

\end{document}